\documentclass[11pt]{amsart}
\usepackage{graphicx,subcaption}
\usepackage{amsmath} 
\usepackage{amsthm,amsfonts,amssymb,mathrsfs,amscd,amstext,amsbsy} 
\usepackage{epic,eepic} 
\usepackage{yfonts}
\usepackage{paralist,enumerate}
\usepackage[all]{xy}
\usepackage{hyperref}
\hypersetup{colorlinks}
\usepackage{tikz}
\usetikzlibrary {positioning}
\definecolor {processblue}{cmyk}{0.96,0,0,0}
\usetikzlibrary{arrows}
\tikzset{    vertex/.style={circle,draw,minimum size=1.5em},    edge/.style={->,> = latex'}}

\newtheorem{theorem}{Theorem}[section]
\newtheorem{cor}[theorem]{Corollary}
\newtheorem{lemma}[theorem]{Lemma}

\newtheorem{theo}[theorem]{Theorem}
\newtheorem{lem}[theorem]{Lemma}
\newtheorem{pro}[theorem]{Proposition}
\newtheorem{rem}[theorem]{Remark}
\newtheorem{exa}[theorem]{Example}
\newtheorem{que}[theorem]{Question}

\newtheorem{Definition}[theorem]{Definition}
\newtheorem*{Definition*}{Definition}

\def\qed{\hfill \ifhmode\unskip\nobreak\fi\quad\ifmmode\Box\else$\Box$\fi\\ }

\begin{document}

\title[4-dimensional almost complex $S^1$-manifolds]{Circle actions on four dimensional almost complex manifolds with discrete fixed point sets}
\author{Donghoon Jang}
\address{Department of Mathematics, Pusan National University, Pusan, Korea}
\email{donghoonjang@pusan.ac.kr}

\begin{abstract}
We establish a necessary and sufficient condition for pairs of integers to arise as the weights at the fixed points of an effective circle action on a compact almost complex 4-manifold with a discrete fixed point set. As an application, we provide a necessary and sufficient condition for a pair of integers to arise as the Chern numbers of such an action, answering negatively a question by Sabatini whether $c_1^2[M] \leq 3 c_2[M]$ holds for any such manifold $M$. We achieve this by demonstrating that pairs of integers that arise as weights of a circle action, also arise as weights of a restriction of a $\mathbb{T}^2$-action. Furthermore, we discuss applications to circle actions on complex/symplectic 4-manifolds and semi-free circle actions with discrete fixed point sets.
\end{abstract}

\maketitle

\tableofcontents

\section{Introduction and statements of main results} \label{s1}

The purpose of this paper is to classify the fixed point data of a circle action on a 4-dimensional compact almost complex manifold with a discrete fixed point set and discuss several applications. By the fixed point data, we mean the collection of multisets of weights at the fixed points. The fixed point data encodes information about a manifold; the Chern numbers, the Euler characteristic, the Hirzebruch $\chi_y$-genus, the signature, the Todd genus, etc.

First, we discuss the known classification results of circle actions on 4-dimensional compact manifolds with discrete fixed point sets in different categories. In the complex case, Carrell, Howard, and Kosniowski proved that a holomorphic vector field on a complex surface can be obtained from $\mathbb{CP}^2$ or a Hirzebruch surface by blow-ups \cite{CHK}. In the symplectic case, a Hamiltonian $S^1$-space can also be obtained through blow-ups from $\mathbb{CP}^2$ or a Hirzebruch surface \cite{AhHa}, \cite{Au}, \cite{Ka}. It is worth noting that these papers also address cases where the fixed point set is not discrete, both for complex manifolds and symplectic manifolds.

In the context of oriented manifolds, Fintushel classified circle actions on 4-dimensional compact oriented manifolds in terms of orbit data \cite{F}. The classification of the fixed point data of a circle action on a 4-dimensional compact oriented manifold with a discrete fixed point set is given in \cite{J4}. For an oriented manifold, the sign of each weight at a fixed point is not well-defined, whereas for an almost complex manifold it is well-defined.

Throughout this paper, we assume that any group action on an almost complex manifold preserves the almost complex structure. Let the circle group $S^1$ act on an almost complex manifold $M$. Let $p$ be an isolated fixed point. Then the local action of $S^1$ near $p$ can be described by 
\begin{center}
$g \cdot (z_1,\cdots,z_n)=(g^{w_{p,1}}z_1,\cdots,g^{w_{p,n}}z_n)$, 
\end{center}
for all $g \in S^1 \subset \mathbb{C}$ and $z_1,\cdots,z_n \in \mathbb{C}$, where $\dim M=2n$ and $w_{p,i}$ are non-zero integers, for $1 \leq i \leq n$. The non-zero integers $w_{p,i}$ are called the \textbf{weights} at $p$. The \textbf{index} $n_p$ of $p$ is the number of negative weights at $p$.

In this paper, we establish a necessary and sufficient condition for a set of pairs of integers to arise as the weights at the fixed points of an effective circle action on a 4-dimensional compact almost complex manifold with a discrete fixed point set.

\begin{theorem} \label{t01}
For $1 \leq i \leq k$, let $\{a_i,b_i\}$ be an unordered pair of relatively prime non-zero integers. There exists an effective circle action on a 4-dimensional compact, connected almost complex manifold $M$ with $k$ fixed points with weights $\{a_i,b_i\}_{i=1}^k$ if and only if the following hold.
\begin{enumerate}[(1)]
\item The number of weights that are $+1$ at fixed points of index $j$ is equal to the number of weights that are $-1$ at fixed points of index $j+1$ for $j \in \{0,1\}$.
\item Given an integer $w > 1$ and a relatively prime integer $x$, the following sets have the same cardinality.
\begin{enumerate}[(a)]
\item Fixed points with one weight $+w$ and the other equal to $x$ modulo $w$.
\item Fixed points with one weight $-w$ and the other equal to $x$ modulo $w$.
\end{enumerate}
\end{enumerate}
Moreover, given a pairing between the fixed points in (2a) and (2b) for each $w$ and $x$, we can choose $M$ so that each pair of fixed points is contained in a two-sphere with generic stabilizer $\mathbb{Z}_w$.
\end{theorem}

In fact, we prove more. To discuss, let $M$ be a 4-dimensional compact almost complex manifold equipped with a circle action having a discrete fixed point set. To $M$, we shall associate a certain type of multigraph $\Gamma$ that contains information on the weights at the fixed points and isotropy 2-spheres. Each vertex corresponds to a fixed point, and each edge has a direction and is labeled by a positive integer, so that if an edge $e$ with label $w$ has an initial vertex $p$ and a terminal vertex $q$, the fixed point $p$ has a weight of $+w$ and the fixed point $q$ has a weight of $-w$; see Definitions \ref{d11} and \ref{d12}. Theorem \ref{t11} states that such a multigraph $\Gamma$ can be obtained from a minimal multigraph for a semi-free action followed by a combination of 4 types of operations on it; moreover, the converse holds that any multigraph obtained from a minimal multigraph followed by the 4 types of operations, encodes the fixed point data of some 4-dimensional compact almost complex $S^1$-manifold with a discrete fixed point set.
In other words, we provide minimal data and operations for the weights at the fixed points of such a manifold $M$.

In \cite{S}, Sabatini posed the following ``geography problem'', what possible values are for the Chern numbers of an almost complex $S^1$-manifold with a discrete fixed point set.

\begin{que} \label{q16} \cite{S} Let the circle act on a compact almost complex manifold $M$ with a discrete fixed point set. Then what are the possible values for the Chern numbers of $M$? \end{que}

In particular, Sabatini inquired if the following inequality holds, as it does when $M$ admits a symplectic structure and the circle action is Hamiltonian.

\begin{que} \label{q17} \cite{S} Let the circle act on a 4-dimensional compact almost complex manifold $M$ with a discrete fixed point set. Then does the inequality $c_1^2[M] \leq 3c_2[M]$ hold? \end{que}

Here, $c_1$ and $c_2$ are the first and second equivariant Chern classes of $M$, respectively, and $[M]$ is the fundamental class of $M$. In \cite[$\S 6.1$]{S}, Sabatini proved that for a 4-dimensional compact almost complex $S^1$-manifold $M$ with a discrete fixed point set, $c_1^2[M]=10 N_0-N_1$ and $c_2[M]=2N_0+N_1$, where $N_i$ denotes the number of fixed points of index $i$. Therefore, proving $c_1^2[M] \leq 3c_2[M]$ is equivalent to proving $N_0 \leq N_1$. We provide a complete answer to Question \ref{q16} in dimension 4, answering Question \ref{q17} negatively by constructing a manifold $M$ with $N_1 < N_0$.

\begin{theo} \label{c115}
Fix $n_1,n_2 \in \mathbb{R}$. There exists a circle action on a 4-dimensional compact connected almost complex manifold $M$ with a non-empty discrete fixed point set such that $c_1^2[M]=10n_1-n_2$ and $c_2[M]=2n_1+n_2$ if and only if $n_1$ and $n_2$ are positive integers.
\end{theo}

Theorem \ref{c115} is obtained as a consequence of the following theorem.

\begin{theo} \label{c19}
Fix $N_0,N_1,N_2 \in \mathbb{Z}$, not all zero. There exists a circle action on a 4-dimensional compact connected almost complex manifold $M$ with exactly $N_0$ fixed points of index $0$, $N_1$ fixed points of index $1$, and $N_2$ fixed points of index $2$ if and only if $N_0=N_2$ and $N_0,N_1,N_2$ are all positive. In this case, the Todd genus of $M$ is equal to $N_0$.
\end{theo}

In particular, Theorem \ref{c19} implies that any $S^1$-action on a 4-dimensional compact almost complex manifold with fixed points has at least $2 \cdot \mathrm{Todd}(M)+1$ fixed points, and has at least three fixed points. Here, $\mathrm{Todd}(M)$ is the Todd genus of $M$.
In fact, three fixed points can only occur in dimension 4; see \cite{J1, J2}.

Moreover, for a set of pairs of integers that arises as the weights at the fixed points of an effective $S^1$-action on a compact almost complex 4-manifold $M$, we construct another compact almost complex 4-manifold $M'$ equipped with a $\mathbb{T}^2$-action, in which the set arises as the weights at the fixed points of the action of the first $S^1$-factor of $\mathbb{T}^2$ on $M'$.

\begin{theo} \label{t314} Let the circle act effectively on a 4-dimensional compact almost complex manifold $M$ with a discrete fixed point set. Then there exists another 4-dimensional compact, connected almost complex manifold $M'$ equipped with a $\mathbb{T}^2$-action, such that the action of the first $S^1$-factor of $\mathbb{T}^2$ on $M'$ has the same weights at the fixed points as $M$. \end{theo}

Let $M$ be a 4-dimensional compact connected manifold, complex or symplectic. Let the circle act on $M$, preserving the given structure. Suppose that the fixed point set is non-empty and discrete. If $M$ is complex, in \cite{CHK}, Carrell, Howard, and Kosniowski classified the fixed point data of $M$. If $M$ is symplectic, the fixed point data is classified by Ahara-Hattori \cite{AhHa}, Audin \cite{Au}, and Karshon \cite{Ka}. When $M$ is complex or symplectic, the Todd genus of $M$ is equal to 1. The Todd genus $\textrm{Todd}(M)$ of $M$ is equal to the number of fixed points of index 0; see Theorem \ref{t21}. Therefore, as an application of Theorem \ref{t01}, we recover the classifications of the fixed point data, of a circle action on a compact complex (or symplectic) 4-manifold with a discrete fixed point set.

\begin{theo} \label{t02}
For $1 \leq i \leq k$, let $\{a_i,b_i\}$ be an unordered pair of relatively prime non-zero integers. There exists an effective circle action on a 4-dimensional compact, connected manifold $M$, complex or symplectic, with $k$ fixed points with weights $\{a_i,b_i\}_{i=1}^k$ if and only if the following hold.

\begin{enumerate}[(1)]
\item The number of weights that are $+1$ at fixed points of index $j$ is equal to the number of weights that are $-1$ at fixed points of index $j+1$ for $j \in \{0,1\}$.
\item Given an integer $w > 1$ and a relatively prime integer $x$, the following sets have the same cardinality.
\begin{enumerate}[(a)]
\item Fixed points with one weight $+w$ and the other equal to $x$ modulo $w$.
\item Fixed points with one weight $-w$ and the other equal to $x$ modulo $w$.
\end{enumerate}
\item There is exactly one fixed point whose weights are all positive.
\end{enumerate}
\end{theo}

We discuss semi-free circle actions on compact almost complex manifolds with discrete fixed point sets. 
In \cite{TW}, Tolman and Weitsman proved that if the circle acts semi-freely on a $2n$-dimensional compact almost complex manifold with a discrete fixed point set, the number of fixed points of index $i$, denoted as $N_i$, satisfies $N_i=N_0 \cdot {n \choose i}$ for $0 \leq i \leq n$. As a consequence, they proved that if the circle acts symplectically and semi-freely on a compact symplectic manifold $M$ with a non-empty discrete fixed point set, the action must be Hamiltonian (and hence the Todd genus of $M$ is equal to 1), and the fixed point data matches that of a diagonal action on the product of $S^2$'s, on each of which the circle acts by rotation. In particular, there are exactly $2^n$ fixed points. Feldman \cite{Fe} and Li \cite{L} reproved this. In this paper, we prove the converse of the above statement; given any positive integers $k$ and $n>1$, we construct a semi-free circle action on a $2n$-dimensional compact connected almost complex manifold $M$ with a discrete fixed point set such that $\textrm{Todd}(M)=k$ and $M$ has precisely $k \cdot 2^n$ fixed points. Therefore, our result can be viewed as a generalization of the result for semi-free symplectic circle actions by \cite{TW} to almost complex manifolds.

\begin{theo} \label{t13}
Fix non-negative integers $N_0,N_1,\cdots,N_n$, where $n \geq 2$. There exists a semi-free circle action on a $2n$-dimensional compact, connected almost complex manifold $M$ with exactly $N_i$ fixed points of index $i$ for each $i$ if and only if $N_i=N_0 {n \choose i}$ for all $i$; moreover, in this case, the Todd genus of $M$ is $N_0$.
\end{theo}

In Theorem \ref{t13}, we exclude dimension 2 simply because, in dimension 2, if there is a fixed point, the manifold is the 2-sphere $S^2$ and it has 2 fixed points, that is, a rotation of the 2-sphere.

The paper is organized as follows. 
\begin{itemize}
\item In Section \ref{s2}, we review background and provide preliminary results on almost complex $S^1$-manifolds with discrete fixed point sets. 
\item Section \ref{s3} presents a minimal multigraph and operations for a labeled directed multigraph that contains the fixed point data of a circle action on a 4-dimensional compact almost complex manifold with a discrete fixed point set (Theorem \ref{t11}).
\item Section \ref{plumbing} discusses equivariant plumbing techniques used to construct a $\mathbb{T}^2$-action on a 4-dimensional compact connected almost complex manifold with a discrete fixed point set.
\item Section \ref{extends} introduces the concept of extendability of a 2-regular labeled directed multigraph and explores related properties.
\item In Section \ref{s7}, using the preliminaries from the previous sections, we prove Theorems \ref{t01}, \ref{t314}, and \ref{t11}.
\item Section \ref{s5} is dedicated to proving Theorem \ref{c115} and Theorem \ref{c19}.
\item Section \ref{s4} involves a comparison of circle actions with discrete fixed point sets on 4-dimensional almost complex, complex, and symplectic manifolds, and a proof of Theorem \ref{t02}.
\item In Section \ref{s6}, we discuss semi-free actions with discrete fixed point sets on symplectic and almost complex manifolds, and prove Theorem \ref{t13}
\end{itemize}

\section{Background and preliminaries} \label{s2}

In this section, we review background and present preliminary results.

For an action of a group $G$ on a manifold $M$, we denote by $M^G$ the set of points in $M$ that are fixed by the $G$-action. That is, $M^G=\{m \in M \, | \, g \cdot m = m \textrm{ for all } g \in G\}$.

An \textbf{almost complex manifold} $(M,J)$ is a manifold $M$ equipped with a fiberwise linear complex structure $J_m$ on each tangent space $T_mM$. An action of a group $G$ on an almost complex manifold $(M,J)$ is said to \textbf{preserve the almost complex structure} if $dg \circ J=J \circ dg$ for all $g \in G$.

Let $M$ be a compact almost complex manifold. The Hirzebruch $\chi_y$-genus is the genus belonging to the power series $\frac{x(1+ye^{-x(1+y)})}{1-e^{-x(1+y)}}$. The Hirzebruch $\chi_y$-genus of $M$ contains three pieces of information; $\chi_{-1}(M)=\chi(M)$ is the Euler characteristic of $M$, $\chi_0(M)=\textrm{Todd}(M)$ is the Todd genus of $M$, and $\chi_1(M)=\textrm{sign}(M)$ is the signature of $M$. In \cite{HBJ}, Hirzebruch, Berger, and Jung proved that for a circle action on a compact complex manifold with a discrete fixed point set, the Hirzebruch $\chi_y$-genus is rigid under the circle action. Following the idea, Li extended the result to almost complex manifolds.

\begin{theorem} \label{t21} \cite{HBJ}, \cite{L} Let the circle act on a $2n$-dimensional compact almost complex manifold $M$ with a discrete fixed point set. For each integer $i$ such that $0 \leq i \leq n$,
\begin{center}
$\displaystyle \chi^i(M)=\sum_{p \in M^{S^1}} \frac{\sigma_i (t^{w_{p,1}}, \cdots, t^{w_{p,n}})}{\prod_{j=1}^n (1-t^{w_{p,j}})} = (-1)^i N_i = (-1)^i N_{n-i}$,
\end{center}
where $\chi_y(M)=\sum_{i=0}^n \chi^i(M) \cdot y^i$ is the Hirzebruch $\chi_y$-genus of $M$, $t$ is an indeterminate, $\sigma_i$ is the $i$-th elementary symmetric polynomial in $n$ variables, and $N_i$ is the number of fixed points of index $i$. Moreover, $\chi^0(M)=\chi_0(M)$ is equal to the Todd genus of $M$. \end{theorem}

For a circle action on a compact almost complex manifold with a discrete fixed point set, each time the weight $w$ occurs, the weight $-w$ also occurs.

\begin{lem} \cite{H}, \cite{L} Let the circle act on a $2n$-dimensional compact almost complex manifold $M$ with a discrete fixed point set. For each integer $w$,
\begin{center}
$\displaystyle \sum_{p \in M^{S^1}} N_p(w)=\sum_{p \in M^{S^1}} N_p(-w)$.
\end{center}
Here, for an integer $w$, $N_p(w)=|\{i : w_{p,i}=w, 1 \leq i \leq n\}|$ denotes the number of times the weight $w$ occurs at $p$.
\end{lem}

In addition, the smallest positive weight satisfies a stronger property.

\begin{pro} \label{p22} \cite{JT} Let the circle act  on a $2n$-dimensional compact almost complex manifold $M$ with a discrete fixed point set. Let $a$ be the smallest positive weight that occurs at any fixed point. Given any $j \in \{0,1,\dots,n-1\}$, the number of times the weight $+a$ occurs at fixed points of index $j$ is equal to the number of times the weight $-a$ occurs at fixed points of index $j+1$. That is, for any $j \in \{0,1,\dots,n-1\}$,
\begin{center}
$\displaystyle \sum_{p \in M^{S^1}, \, n_p=j} N_p(a)=\sum_{p \in M^{S^1}, \, n_p=j+1} N_p(-a)$,
\end{center}
where $n_p$ is the index of $p$, and for an integer $w$, $N_p(w)=|\{i : w_{p,i}=w, 1 \leq i \leq n\}|$ is the number of times the weight $w$ occurs at $p$. \end{pro}

Applying Proposition \ref{p22} to the smallest positive weight, it follows that there must exist two fixed points whose indices differ by 1.

\begin{cor} \cite{J5} \label{c23} Let the circle act on a $2n$-dimensional compact almost complex manifold $M$ with a non-empty discrete fixed point set. Then there exists $0 \leq i \leq n-1$ such that both of $N_i$ and $N_{i+1}$ are not zero, where $N_i$ is the number of fixed points of index $i$. \end{cor}

Alternatively, since $\chi^i(M)=(-1)^iN_i$ for $0 \leq i \leq n$, Corollary \ref{c23} implies that there exists $0 \leq i \leq n-1$ such that both of $\chi^i(M)$ and $\chi^{i+1}(M)$ are not zero, where $\chi_y(M)=\sum_{i=0}^n \chi^i(M) \cdot y^i$ is the Hirzebruch $\chi_y$-genus of $M$.

The weights at fixed points lying in the same connected component of $M^{\mathbb{Z}_w}$ are equal modulo $w$;

\begin{lem} \label{l21} \cite{T}, \cite{GS} Let the circle act on a compact almost complex manifold $M$. Let $p$ and $p'$ be fixed points that lie in the same connected component of $M^{\mathbb{Z}_{w}}$, for some positive integer $w$. Then the $S^{1}$-weights at $p$ and at $p'$ are equal modulo $w$. \end{lem}

Here, the group $\mathbb{Z}_w$ acts on $M$ as a cyclic subgroup of $S^1$. Lemma \ref{l21} states that if $p$ and $p'$ lie in the same connected component of $M^{\mathbb{Z}_w}$, there exists a bijection $\pi:\{1,2,\cdots,n\} \to \{1,2,\cdots,n\}$ such that $w_{p,i} \equiv w_{p',{\pi(i)}} \mod w$ for all $i$, where $\dim M=2n$.

For a circle action on a compact almost complex manifold with a discrete fixed point set, there is a certain type of multigraph that contains information of the data on the fixed point set.

\begin{Definition} \label{d11}
\begin{enumerate}[(1)]
\item A \textbf{multigraph} is a pair $(V,E)$, where $V$ is a set of vertices and $E$ is a set of (unordered) edges.
\item A multigraph $(V,E)$ is called \textbf{directed} if there are maps $i:E \to V$ and $t:E \to V$ giving the initial and terminal vertices of each edge.
\item A multigraph $(V,E)$ is called \textbf{labeled} if there is a map $w$ from $E$ to the set $\mathbb{N}^+$ of positive integers. For any edge $e$, $w(e)$ is called the \textbf{label} of $e$.
\item Let $\Gamma=(V,E)$ be a labeled directed multigraph. 
\begin{enumerate}[(a)]
\item Let $e$ be an edge of $\Gamma$. We say that $(i(e),t(e))$ is \textbf{$w(e)$-edge}. Alternatively, we say that $(t(e),i(e))$ is $(-w(e))$-edge.
\item For a vertex $v$ of $\Gamma$, we call $\{w(e) \ | \ i(e)=v\} \cup \{-w(e) \ | \ t(e)=v\}$ the \textbf{weights} at $v$.
\end{enumerate}
\end{enumerate}
\end{Definition}

\begin{Definition} \label{d12} \cite{JT}
Let the circle act on a compact almost complex manifold $M$ with a discrete fixed point set. We say that a labeled directed multigraph $\Gamma$ \textbf{describes} $M$ if the following hold:
\begin{enumerate}[(1)]
\item The vertex set of $\Gamma$ is the fixed point set $M^{S^1}$ of $M$.
\item The multiset of the weights at $p$ is $\{w(e)\,|\,i(e)=p\} \cup \{-w(e)\,|\,t(e)=p\}$ for all $p \in M^{S^1}$.
\item For each edge $e$, the two endpoints $i(e)$ and $t(e)$ lie in the same connected component of the isotropy submanifold $M^{\mathbb{Z}/(w(e))}$. Here, $\mathbb{Z}/(w(e))$ acts on $M$ as a subgroup of $S^1$.
\end{enumerate}
\end{Definition}

Suppose that a labeled directed multigraph $\Gamma$ describes a compact almost complex $S^1$-manifold $M$ with a discrete fixed point set.
If an edge $e$ with label $w$ has initial vertex $p$ and terminal vertex $q$, Definition \ref{d12} means that $p$ has the weight $+w$ and $q$ has the weight $-w$; moreover, $p$ and $q$ are in the same component of $M^{\mathbb{Z}_w}$. Thus, the multigraph $\Gamma$ contains information on the weights at the fixed points. In addition, if $\dim M = 2n$, $\Gamma$ is $n$-regular; every vertex has $n$-edges.

For a directed multigraph, by a self-loop we mean an edge $e$ whose initial vertex and terminal vertex coincide, that is, $i(e)=t(e)$. 
In \cite{JT}, using Proposition \ref{p22}, Tolman and the author proved that for a compact almost complex $S^1$-manifold $M$ with a discrete fixed point set, there exists a labeled directed multigraph describing $M$.

\begin{lem} \label{l26} \cite{JT}
Let the circle act on a compact almost complex manifold $M$ with a discrete fixed point set. Then there exists a labeled, directed multigraph describing $M$ such that, for each edge $e$, the index of $i(e)$ in the isotropy submanifold $M^{\mathbb{Z}/(w(e))}$ is one less than the index of $t(e)$ in $M^{\mathbb{Z}/(w(e))}$. In particular, the multigraph has no self-loops. \end{lem}

We shall explore a multigraph satisfying the conditions in Lemma \ref{l26} in dimension 4. Consider a circle action on a 4-dimensional compact almost complex manifold $M$ with a discrete fixed point set. Without loss of generality, by quotienting out by the subgroup that acts trivially, we may assume that the action is effective. Suppose that a fixed point $p$ has a weight $+w$ for some positive integer $w$ such that $w>1$. The group $\mathbb{Z}_w$ acts on $M$ as a subgroup of $S^1$. The set $M^{\mathbb{Z}_w}$ of points in $M$ that are fixed by the $\mathbb{Z}_w$-action is a union of smaller dimensional compact almost complex submanifolds. Let $F$ be a connected component of $M^{\mathbb{Z}_w}$ that contains $p$. Then $\dim F=2m$ if and only if $p$ has exactly $m$ weights that are divisible by $w$. Since the action is effective, $\dim M=4$, and $p$ has the weight $w>1$, it follows that $\dim F=2$. The circle group acts on $F$ as a restriction, and the $S^1$-action on $F$ has $p$ as a fixed point. This implies that $F$ is a 2-sphere, since among compact oriented Riemann surfaces, only $S^2$ admits a circle action with a fixed point. Since any circle action on $S^2$ has 2 fixed points, the $S^1$-action on $F$ has another fixed point $q$ that has $-w$ as a weight of the $S^1$-action; hence the original circle action on $M$ also has a weight $-w$ at $q$. Draw an edge from $p$ to $q$ with label $w$. By (4) of Definition \ref{d11}, $(p,q)$ is $w$-edge. Alternatively, $(q,p)$ is $(-w)$-edge. In \cite{Ka}, the $(p,q)$-edge is called a $\mathbb{Z}_w$-sphere.

By Proposition \ref{p22}, the number of times the weight $+1$ occurs at fixed points of index $i$ is equal to the number of times the weight $-1$ occurs at fixed points of index $i+1$, for $i \in \{0,1\}$. Therefore, each time some fixed point $p$ has the weight $+1$, there exists another fixed point $q$ that has the weight $-1$ with $n_p+1=n_q$; draw an edge from $p$ to $q$, giving a label $1$. This is how we draw a 2-regular labeled directed multigraph describing $M$ with no self-loops in Lemma \ref{l26}.

The above discussion implies the below lemma.

\begin{lem} \label{2sphere}
Let the circle act effectively on a 4-dimensional compact almost complex manifold $(M,J)$ with a discrete fixed point set. Let $\Gamma$ be a labeled directed multigraph describing $M$. Let $w>1$ be an integer. Suppose that $(p,q)$ is $w$-edge of $\Gamma$. Then $p$ and $q$ lie in the same 2-sphere $(S^2,J)$, which is a component of $M^{\mathbb{Z}_w}$. The circle acts on $(S^2,J)$ by rotating $w$-times, having fixed points $p$ and $q$ with weights $w$ and $-w$ for this action.
\end{lem}

Thus, any edge with label bigger than 1 corresponds to an isotropy 2-sphere. Next, we discuss properties of the multigraph in Lemma \ref{l26} that we will use later. For this, we introduce more terminologies.

\begin{Definition}
\begin{enumerate}[(1)]
\item Let $\Gamma=(V,E)$ be a directed multigraph. For a vertex $v$ of $\Gamma$, the \textbf{index} $n_v$ of $v$ is the number of edges whose terminal vertex is $v$. The \textbf{Todd genus} $T(\Gamma)$ of $\Gamma$ is the number of vertices of index 0.
\item A labeled directed multigraph $\Gamma$ is called \textbf{effective}, if for each vertex $v$, the greatest common divisor of the labels of the edges of $v$ is equal to 1.
\item An $n$-regular labeled directed multigraph is said to satisfy the \textbf{equal modulo property}, if $(p,q)$ is $w$-edge, then the weights at $p$ and the weights at $q$ are equal modulo $w$.
\item A 2-regular effective labeled directed multigraph $\Gamma$ is called \textbf{admissible} if for every edge $e$ with label $1$, the index of the initial vertex of $e$ is one less than the index of the terminal vertex of $e$, that is, $n_{i(e)}+1=n_{t(e)}$.
\end{enumerate}
\end{Definition}

Note that a 2-regular admissible effective directed labeled multigraph necessarily has no self-loops. We show that the multigraph in Lemma \ref{l26} is admissible and effective and satisfies the equal modulo property.

\begin{lem} \label{graphequal}
Let the circle act effectively on a 4-dimensional compact almost complex manifold $M$ with a discrete fixed point set. Then there exists a labeled directed multigraph describing $M$ that is admissible and effective and satisfies the equal modulo property.
\end{lem}

\begin{proof}
By Lemma \ref{l26}, there exists a labeled directed multigraph $\Gamma$ describing $M$ such that, for each edge $e$, the index of $i(e)$ in the isotropy submanifold $M^{\mathbb{Z}/(w(e))}$ is one less than the index of $t(e)$ in $M^{\mathbb{Z}/(w(e))}$. 

Let $e$ be an edge with label 1. Then $\mathbb{Z}/(w(e))$ is the trivial group and thus $M^{\mathbb{Z}/(w(e))}=M$. Therefore, the index of the vertex $i(e)$, which is the index of the fixed point $i(e)$ in $M$, is one less than the index of the vertex $t(e)$, which is the index of the fixed point $t(e)$ in $M$. Thus, $\Gamma$ is admissible.

Since the action is effective, the greatest common divisor of the weights at each fixed point is 1 and hence $\Gamma$ is effective.

We show that $\Gamma$ satisfies the equal modulo property. Let $e$ be an edge. If the label $w(e)$ of $e$ is 1, then the weights at $i(e)$ and the weights at $t(e)$ are equal modulo $w(e)$. Suppose that $w(e)>1$. Since $\Gamma$ describes $M$, the fixed points $i(e)$ and $t(e)$ lie in the same component of $M^{\mathbb{Z}/(w(e))}$. By Lemma \ref{l21}, the weights at $i(e)$ and the weights at $t(e)$ are equal modulo $w(e)$.
\end{proof}

\begin{rem}
One may extend the definition of admissibility for a 2-regular effective labeled directed multigraph to one for an $n$-regular effective labeled directed multigraph; then the proof of Lemma \ref{graphequal} applies to an arbitrary dimensional compact almost complex $S^1$-manifold with a discrete fixed point set.
\end{rem}

The following lemma shows that we can equivariantly glue together two 4-dimensional compact connected almost complex manifolds equipped with $\mathbb{T}^2$-actions along their free orbits. For a $\mathbb{T}^2$-action on an almost complex manifold $M$, we denote by $\Sigma_M^{\mathbb{T}^2}$ the collection of multisets of the $\mathbb{T}^2$-weights at the fixed points of $M$.

\begin{lem} \label{l27} For $i=1,2$, let $M_i$ be a 4-dimensional compact connected almost complex manifold equipped with an effective $\mathbb{T}^2$-action. Then we can perform equivariant connected sum along free orbits of $M_i$ to construct a 4-dimensional compact connected almost complex manifold $M$ equipped with an effective $\mathbb{T}^2$-action, such that the action has $\Sigma_{M_1}^{\mathbb{T}^2} \sqcup \Sigma_{M_2}^{\mathbb{T}^2}$ as the collection of multisets of the $\mathbb{T}^2$-weights at the fixed points of $M$.

In addition, suppose that the following hold:
\begin{enumerate}
\item For $i=1,2$, for any fixed point $p$ in $M_i$, the weights at $p$ are not zero for the action of the first $S^1$-factor of $\mathbb{T}^2$ on $M_i$.
\item For $i=1,2$, a labeled directed multigraph $\Gamma_i$ describes $M_i$ for the action of the first $S^1$-factor of $\mathbb{T}^2$ on $M_i$. 
\end{enumerate}
Then the disjoint union of the two multigraphs $\Gamma_1$ and $\Gamma_2$ describes $M$ for the action of the first $S^1$-factor of $\mathbb{T}^2$. \end{lem}

\begin{proof} For each $i=1,2$, take an equivariant tubular neighborhood $U_i \approx D^2 \times \mathbb{T}^2$ of a free orbit of the $\mathbb{T}^2$-action on $M_i$. For each $i=1,2$, the almost complex structure $J_i$ on $M_i$ is $\mathbb{T}^2$-invariant, and hence on $U_i$ and on the boundary of $U_i$. Denote by $B_i =\partial U_i \approx S^1 \times \mathbb{T}^2$ the boundary of $U_i$.

For $i=1,2$, the complex structure on $B_i$ is classified by a continuous map from $B_i$ to $GL(4,\mathbb{R})/GL(2,\mathbb{C})$, which is homotopy equivalent to $S^2 \sqcup S^2$. Since $\mathbb{T}^2$ acts on $B_i$ freely and $J_i$ invariantly, the orbit space is $S^1$ and the map factors through $B_i \to S^1 \to GL(4,\mathbb{R})/GL(2,\mathbb{C})$. Choose a circle $S$ in the interior of $D^2$ of $U_2 \approx D^2 \times \mathbb{T}^2$ and let $C_2 \approx S \times \mathbb{T}^2 \subset U_2$.

We glue $B_1$ and $C_2$ equivariantly. Denote by $\phi$ the equivariant diffeomorphism between $B_1$ and $C_2$. By reversing the orientation of $M_2$ if necessary, the complex structure $J'$ on $C_2$ induced by $J_1$ via the equivariant diffeomorphism $\phi$ and the complex structure $J_2$ on $B_2$ lie in the same component of $GL(4,\mathbb{R})/GL(2,\mathbb{C}) \approx S^2 \sqcup S^2$. Therefore, between $C_2$ and $B_2$, we can perturb the almost complex structure $J_2$ on $B_2$ to agree with the induced complex structure $J'$ on $C_2$. Therefore, when we take the equivariant connected sum of $M_1$ and $M_2$ by gluing $B_1$ and $C_2$ equivariantly, the resulting manifold $M$ admits an almost complex structure.

On the other hand, since the equivariant connected sum is along free orbits of $M_1$ and $M_2$, it follows that the fixed point set of the $\mathbb{T}^2$-action on $M$ is the union $M_1^{\mathbb{T}^2} \sqcup M_2^{\mathbb{T}^2}$ of the fixed point sets of $M_1$ and $M_2$, and the weights at the fixed points of the $\mathbb{T}^2$-action on $M$ is $\Sigma_{M_1}^{\mathbb{T}^2} \sqcup \Sigma_{M_2}^{\mathbb{T}^2}$. This proves the first claim.

The second claim follows from the construction of $M$. For $i=1,2$, for any fixed point $p$ in $M_i$, the weights at $p$ are not zero for the action of the first $S^1$-factor of $\mathbb{T}^2$ on $M_i$. Then the action of the first $S^1$-factor of $\mathbb{T}^2$ on $M$ has the same fixed point set as the $\mathbb{T}^2$-action $M$, that is, $M^{S^1}=M^{\mathbb{T}^2}$. Hence if $\Gamma_i$ describes $M_i$ for the action of the first $S^1$-factor of $\mathbb{T}^2$-action on $M_i$ for $i=1,2$, then $\Gamma_1 \sqcup \Gamma_2$ describes $M$ for the action of the first $S^1$-factor of $\mathbb{T}^2$-action on $M$. \end{proof}

\section{Classification of fixed point data in terms of multigraphs} \label{s3}

In this section, we present a minimal multigraph and operations for a multigraph describing a circle action on a 4-dimensional compact almost complex manifold with a discrete fixed point set (Theorem \ref{t11}).
We shall establish preliminary results in Sections \ref{plumbing} and \ref{extends} and then prove Theorem \ref{t11} in Section \ref{s7}.

A group action on a manifold is called \textbf{semi-free}, if the action is free outside the fixed point set. For a semi-free circle action on an almost complex manifold with a discrete fixed point set, every weight at a fixed point is either $+1$ or $-1$. Accordingly, we introduce a notion for a multigraph for a semi-free action.

\begin{Definition} \label{d13}
A labeled directed multigraph $\Gamma$ is called \textbf{semi-free} if every edge of $\Gamma$ has label $1$.
\end{Definition}

By definition, if $\Gamma$ is a 2-regular admissible semi-free multigraph with $T(\Gamma)=k$, then $\Gamma$ has $k$ vertices of index 0, $2k$ vertices of index 1, and $k$ vertices of index 2. Figure \ref{fig1} is an example of a connected 2-regular admissible semi-free multigraph $\Gamma$ with $T(\Gamma)=2$. The following theorem states that a multigraph describing a 4-dimensional compact almost complex $S^1$-manifold with a discrete fixed point set can be obtained from a minimal multigraph (a 2-regular admissible semi-free multigraph) and four types of operations on a multigraph.

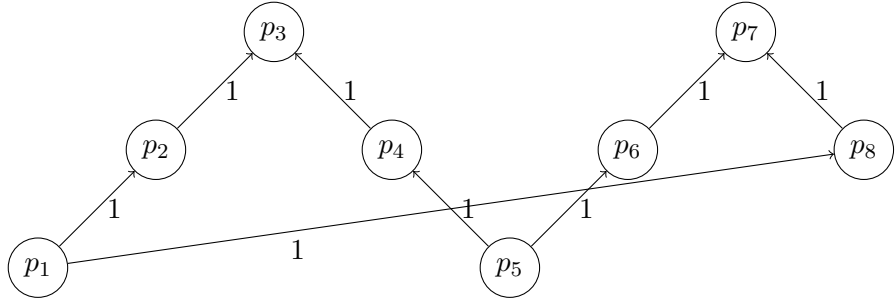
\begin{figure}
\centering
\begin{subfigure}[b][3.5cm][s]{.9\textwidth}
\centering
\vfill
\begin{tikzpicture}[state/.style ={circle, draw}]
\node[state] (A) {$p_1$};
\node[state] (B) [above right=of A] {$p_2$};
\node[state] (C) [above right=of B] {$p_3$};
\node[state] (D) [below right=of C] {$p_4$};
\node[state] (E) [below right=of D] {$p_5$};
\node[state] (F) [above right=of E] {$p_6$};
\node[state] (G) [above right=of F] {$p_7$};
\node[state] (H) [below right=of G] {$p_8$};
\path (A) [->] edge node[right] {$1$} (B);
\path (B) [->] edge node [right] {$1$} (C);
\path (D) [->] edge node [right] {$1$} (C);
\path (E) [->] edge node [right] {$1$} (D);
\path (E) [->] edge node [right] {$1$} (F);
\path (F) [->] edge node [right] {$1$} (G);
\path (H) [->] edge node [right] {$1$} (G);
\path (A) [->] edge node [pos=.3, below] {$1$} (H);
\end{tikzpicture}
\vspace{\baselineskip}
\end{subfigure}\qquad
\caption{2-regular connected semi-free multigraph $\Gamma$ with $T(\Gamma)=2$}\label{fig1}
\end{figure}

\begin{theo} \label{t11} Let the circle act effectively on a 4-dimensional compact almost complex manifold $M$ with a discrete fixed point set. Then a labeled directed multigraph describing $M$ can be constructed in the following way: begin with a disjoint union of connected 2-regular admissible semi-free multigraphs $\Gamma_j$, $1 \leq j \leq k$, for some positive integer $k$. Then, apply a combination of the following 4 operations to each $\Gamma_j$:
\begin{enumerate}[(1)]
\item Suppose that $(p_{i-1},p_i)$ is $a$-edge and $(p_i,p_{i+1})$ is $b$-edge, for some vertices $p_{i-1}$, $p_i$, $p_{i+1}$ and for some positive integers $a$, $b$. Then replace $p_i$ with $(a+b)$-edge $(p_i',p_i'')$ so that $(p_{i-1},p_i')$ is $a$-edge and $(p_i'',p_{i+1})$ is $b$-edge (Figure \ref{fig2}).
\item Suppose that $(p_i,p_{i-1})$ is $c$-edge, $(p_{i+1},p_i)$ is $d$-edge, and $(p_{i+1},p_{i+2})$ is $c$-edge, for some vertices $p_{i-1}$, $p_i$, $p_{i+1}$, $p_{i+2}$ and for some positive integers $c$, $d$. Then replace the label of the $(p_{i+1},p_i)$-edge with $d+c$ (Figure \ref{fig3}).
\item Suppose that $(p_{i-1},p_i)$ is $e$-edge, $(p_i,p_{i+1})$ is $f$-edge, and $(p_{i+2},p_{i+1})$ is $e$-edge, for some vertices $p_{i-1}$, $p_i$, $p_{i+1}$, $p_{i+2}$, and for some positive integers $e$, $f$. Then replace the label of the $(p_i,p_{i+1})$-edge with $f+e$ (Figure \ref{fig4}).
\item Suppose that $(p_{i-1},p_{i-2})$ is $g$-edge, $(p_{i-1},p_i)$ is $h$-edge, $(p_{i},p_{i+1})$ is $g$-edge, and $(p_{i+2},p_{i+1})$ is $h$-edge, for some vertices $p_{i-1}$, $p_i$, $p_{i+1}$, $p_{i+2}$, and for some positive integers $g$, $h$. Then replace $(p_{i-1},p_i)$-edge and $(p_i,p_{i+1})$-edge with $(p_{i-1},p_{i+1})$-edge, removing $p_i$; the new $(p_{i-1},p_{i+1})$-edge has label $(g+h)$ (Figure \ref{fig5}).
\end{enumerate}
The Todd genus of $M$ is $\textrm{Todd}(M)=\sum_{j=1}^k T(\Gamma_j)$. Moreover, given any 2-regular labeled directed multigraph $\Gamma$ constructed from a 2-regular admissible semi-free multigraph followed by a combination of the 4 operations in this theorem, there exists a 4-dimensional compact connected almost complex manifold $M$ equipped with a circle action having a discrete fixed point set that is described by $\Gamma$.
\end{theo}

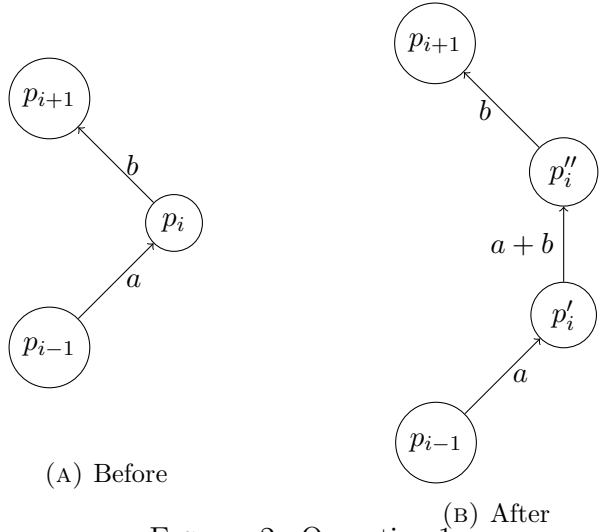
\begin{figure}
\centering
\begin{subfigure}[b][6.5cm][s]{.4\textwidth}
\centering
\vfill
\begin{tikzpicture}[state/.style ={circle, draw}]
\node[state] (a) {$p_{i-1}$};
\node[state] (b) [above right=of a] {$p_i$};
\node[state] (c) [above left=of b] {$p_{i+1}$};
\path (a) [->] edge node[right] {$a$} (b);
\path (b) [->] edge node [right] {$b$} (c);
\end{tikzpicture}
\vfill
\caption{Before}\label{fig2-1}
\end{subfigure}
\begin{subfigure}[b][6.5cm][s]{.4\textwidth}
\centering
\vfill
\begin{tikzpicture}[state/.style ={circle, draw}]
\node[state] (a) {$p_{i-1}$};
\node[state] (b) [above right=of a] {$p_{i}'$};
\node[state] (c) [above=of b] {$p_i''$};
\node[state] (d) [above left=of c] {$p_{i+1}$};
\path (a) [->] edge node[right] {$a$} (b);
\path (b) [->] edge node [left] {$a+b$} (c);
\path (c) [->] edge node [left] {$b$} (d);
\end{tikzpicture}
\vfill
\caption{After}\label{fig2-2}
\vspace{\baselineskip}
\end{subfigure}\qquad
\caption{Operation 1}\label{fig2}
\end{figure}

\begin{figure}
\centering
\begin{subfigure}[b][4.5cm][s]{.4\textwidth}
\centering
\vfill
\begin{tikzpicture}[state/.style ={circle, draw}]
\node[state] (a) {$p_{i-1}$};
\node[state] (b) [below left=of a] {$p_i$};
\node[state] (c) [below right=of b] {$p_{i+1}$};
\node[state] (d) [above right=of c] {$p_{i+2}$};
\path (b) [->] edge node[right] {$c$} (a);
\path (c) [->] edge node [right] {$d$} (b);
\path (c) [->] edge node [right] {$c$} (d);
\end{tikzpicture}
\vfill
\caption{Before}\label{fig3-1}
\vspace{\baselineskip}
\end{subfigure}\qquad
\begin{subfigure}[b][4.5cm][s]{.4\textwidth}
\centering
\vfill
\begin{tikzpicture}[state/.style ={circle, draw}]
\node[state] (a) {$p_{i-1}$};
\node[state] (b) [below left=of a] {$p_i$};
\node[state] (c) [below right=of b] {$p_{i+1}$};
\node[state] (d) [above right=of c] {$p_{i+2}$};
\path (b) [->] edge node[right] {$c$} (a);
\path (c) [->] edge node [right] {$d+c$} (b);
\path (c) [->] edge node [right] {$c$} (d);
\end{tikzpicture}
\vfill
\caption{After}\label{fig3-2}
\vspace{\baselineskip}
\end{subfigure}\qquad
\caption{Operation 2}\label{fig3}
\end{figure}
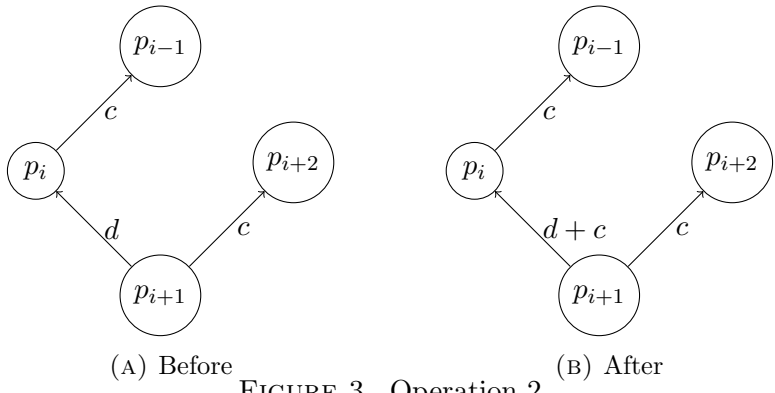

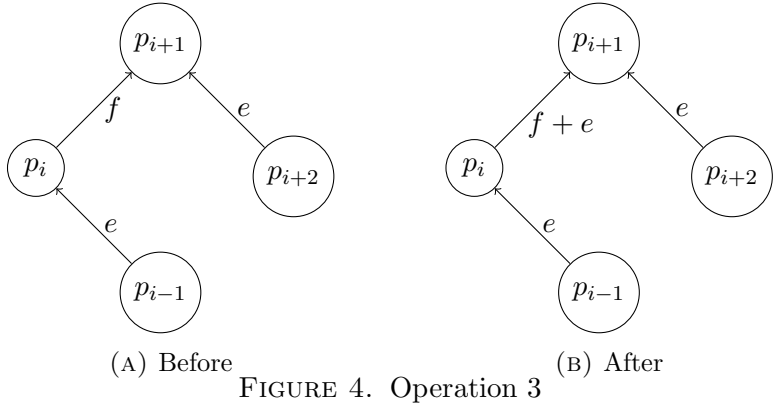
\begin{figure}
\centering
\begin{subfigure}[b][4.5cm][s]{.4\textwidth}
\centering
\vfill
\begin{tikzpicture}[state/.style ={circle, draw}]
\node[state] (a) {$p_{i+2}$};
\node[state] (b) [above left=of a] {$p_{i+1}$};
\node[state] (c) [below left=of b] {$p_{i}$};
\node[state] (d) [below right=of c] {$p_{i-1}$};
\path (a) [->] edge node[right] {$e$} (b);
\path (c) [->] edge node [right] {$f$} (b);
\path (d) [->] edge node [right] {$e$} (c);
\end{tikzpicture}
\vfill
\caption{Before}\label{fig4-1}
\vspace{\baselineskip}
\end{subfigure}\qquad
\begin{subfigure}[b][4.5cm][s]{.4\textwidth}
\centering
\vfill
\begin{tikzpicture}[state/.style ={circle, draw}]
\node[state] (a) {$p_{i+2}$};
\node[state] (b) [above left=of a] {$p_{i+1}$};
\node[state] (c) [below left=of b] {$p_{i}$};
\node[state] (d) [below right=of c] {$p_{i-1}$};
\path (a) [->] edge node[right] {$e$} (b);
\path (c) [->] edge node [pos=.3, right] {$f+e$} (b);
\path (d) [->] edge node [right] {$e$} (c);
\end{tikzpicture}
\vfill
\caption{After}\label{fig4-2}
\vspace{\baselineskip}
\end{subfigure}\qquad
\caption{Operation 3}\label{fig4}
\end{figure}

\begin{figure}
\centering
\begin{subfigure}[b][5.2cm][s]{.4\textwidth}
\centering
\vfill
\begin{tikzpicture}[state/.style ={circle, draw}]
\node[vertex] (a) at (3, 1) {$p_{i+2}$};
\node[vertex] (b) at (1, 2) {$p_{i+1}$};
\node[vertex] (c) at (0, 0) {$p_{i}$};
\node[vertex] (d) at (1, -2) {$p_{i-1}$};
\node[vertex] (e) at (3, -1) {$p_{i-2}$};
\path (a) [->] edge node[below] {$h$} (b);
\path (c) [->] edge node [right] {$g$} (b);
\path (d) [->] edge node [right] {$h$} (c);
\path (d) [->] edge node [above] {$g$} (e);
\end{tikzpicture}
\vfill
\caption{Before}\label{fig5-1}
\vspace{\baselineskip}
\end{subfigure}\qquad
\begin{subfigure}[b][5.2cm][s]{.4\textwidth}
\centering
\vfill
\begin{tikzpicture}[state/.style ={circle, draw}]
\node[vertex] (a) at (3, 1) {$p_{i+2}$};
\node[vertex] (b) at (1, 2) {$p_{i+1}$};
\node[vertex] (d) at (1, -2) {$p_{i-1}$};
\node[vertex] (e) at (3, -1) {$p_{i-2}$};
\path (a) [->] edge node[below] {$h$} (b);
\path (d) [->] edge node [right] {$g+h$} (b);
\path (d) [->] edge node [above] {$g$} (e);
\end{tikzpicture}
\vfill
\caption{After}\label{fig5-2}
\vspace{\baselineskip}
\end{subfigure}\qquad
\caption{Operation 4}\label{fig5}
\end{figure}
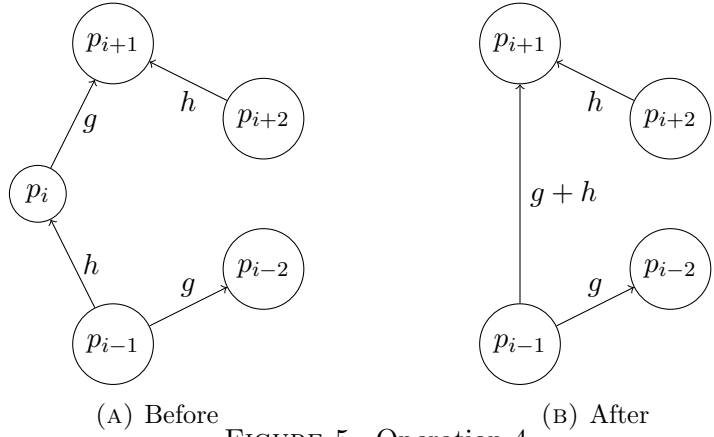

Note that for an effective circle action on a 4-dimensional compact almost complex manifold $M$ with a discrete fixed point set, Theorem \ref{t11} does not imply that there is a unique multigraph describing $M$. Let $w>1$ be an integer. If a fixed point $p$ has the weight $+w$, a fixed point $q$ has the weight $-w$, and $p$ and $q$ lie in the same connected component of $M^{\mathbb{Z}_w}$, then for any multigraph describing $M$, (3) of Definition \ref{d12} means that $(p,q)$ must be $w$-edge. However, if $\pm 1$ occur as weights at some fixed points, there may be many possibilities for drawing edges with label 1. For instance, suppose that a manifold $M$ is described by Figure \ref{fig1}. Then $M$ is also described by a disjoint union of two copies of Figure \ref{fig6} below. However, a multigraph describing $M$ is unique, up to redrawing edges of label 1. Moreover, a multigraph can describe two different manifolds; see Remark \ref{r34}.

\section{Equivariant plumbing} \label{plumbing}

To prove any existence result in this paper for an $S^1$-action on a 4-dimensional almost complex manifold, we shall, in fact, construct a $\mathbb{T}^2$-action on such a manifold. Then, any such result for an $S^1$-action can be attained by restricting the $\mathbb{T}^2$-action to the action of the first $S^1$-factor. To construct this $\mathbb{T}^2$-action, we will utilize equivariant plumbing, which is the primary focus of this section.

Given a sequence of vectors in $\mathbb{Z}^2$ in counterclockwise order, Masuda \cite{M} used the equivariant plumbing to construct a 4-dimensional compact almost complex manifold equipped with a $\mathbb{T}^2$-action. We shall review and reformulate the construction and the proof.

\begin{theo} \label{t31} Let $v_0,v_1,\cdots,v_{k+1}$ be a sequence of vectors in $\mathbb{Z}^2$ in counterclockwise order with $v_k=v_0$ and $v_{k+1}=v_1$. Suppose that 
\begin{enumerate}
\item[(a)] Each successive pair $v_i$ and $v_{i+1}$ is a basis of $\mathbb{Z}^2$, for $1 \leq i \leq k$.
\item[(b)] For $1 \leq i \leq k$, there exists an integer $a_i$ such that $v_{i+1}=-a_i v_i-v_{i-1}$ for each $1 \leq i \leq k$.
\end{enumerate}
Then there exists a 4-dimensional compact, connected almost complex manifold $M$ equipped with an effective $\mathbb{T}^2$-action having $k$ fixed points $p_1,\cdots,p_k$ such that the $\mathbb{T}^2$-weights at $p_i$ are $\{v_i,-v_{i-1}\}$ for $1 \leq i \leq k$.

In addition, suppose that $v_{i,1} \neq 0$ for all $i$, where $v_i=(v_{i,1},v_{i,2}) \in \mathbb{Z}^2$. Then the following hold. 
\begin{enumerate}[(1)]
\item As the first $S^1$-factor of $\mathbb{T}^2$, the $S^1$-action on $M$ has the same fixed point set as the $\mathbb{T}^2$-action, and the $S^1$-weights at $p_i$ are $\{v_{i,1},-v_{i-1,1}\}$ for $1 \leq i \leq k$.
\item Let $\Gamma$ be a labeled directed multigraph such that $(p_i,p_{i+1})$ is $v_{i,1}$-edge for $1 \leq i \leq k$. Then $\Gamma$ describes $M$ for the action of the first $S^1$-factor. Here, $p_{k+1}=p_1$. 
\end{enumerate}

Finally, if the winding number of $v_0,v_1,\cdots,v_{k}$ is one, then there exists such a manifold $M$ that admits a K\"ahler structure preserved by the $\mathbb{T}^2$-action.
\end{theo}

\begin{proof}
Let $a$ be an integer. Let $O(a)$ be the holomorphic line bundle over $\mathbb{CP}^1$ whose self-intersection number of the zero section is $a$. That is, the quotient of $(\mathbb{C}^2-\{0\}) \times \mathbb{C}$ by the $C^*$-action given by 
\begin{center}
$g \cdot (z_1,z_2,w)=(gz_1,gz_2,g^aw)$,
\end{center}
for all $g \in C^*$ and $(z_1,z_2,w) \in (\mathbb{C}^2-\{0\}) \times \mathbb{C}$. Denote by $[z_1,z_2,w]$ the equivalence class of $(z_1,z_2,w)$. Let $\mathbb{T}^2$ act on $O(a)$ by 
\begin{center}
$t \cdot [z_1,z_2,w]=[z_1,t^{u_1}z_2,t^{u_2}w]$
\end{center}
for all $t=(t_1,t_2) \in \mathbb{T}^2$, for some $u_1,u_2 \in \mathbb{Z}^2$ that form a basis of $\mathbb{Z}^2$. Let $D_a(u_1,u_2)$ be the disk bundle of $O(a)$ with this $\mathbb{T}^2$-action. The action has two fixed points $q_1=[1,0,0]$ and $q_2=[0,1,0]$, and the $\mathbb{T}^2$-weights at the fixed points $q_1$ and $q_2$ are $\{u_1,u_2\}$ and $\{-u_1,-au_1+u_2\}$, respectively.

For $D_{a_i}(v_i,-v_{i-1})$, the $\mathbb{T}^2$-weights at $q_{i,2}$ are $\{-v_i,-a_i v_i-v_{i-1} \}$ and for $D_{a_{i+1}}(v_{i+1},-v_i)$, the $\mathbb{T}^2$-weights at $q_{i+1,1}$ are $\{v_{i+1},-v_i\}$, for $1 \leq i \leq k$. By the assumption, for $1 \leq i \leq k$, there exists an integer $a_i$ such that $v_{i+1}=-a_i v_i-v_{i-1}$. Hence the $\mathbb{T}^2$-weights at $q_{i,2}$ and $q_{i+1,1}$ agree for $1 \leq i \leq k$. Therefore, for every $1 \leq i \leq k$, we can equivariantly plumb two manifolds $D_{a_i}(v_i,-v_{i-1})$ and $D_{a_{i+1}}(v_{i+1},-v_i)$ at $q_{i,2}$ and $q_{i+1,1}$. Then we get a 4-dimensional compact connected manifold $N$ with the $\mathbb{T}^2$-action and with boundary $\partial N \simeq S^1 \times \mathbb{T}^2$. As in the proof of Theorem 5.1 of \cite{M}, by pasting $N$ and $D^2 \times \mathbb{T}^2$ along the boundary of $N$, we get a 4-dimensional closed connected almost complex $\mathbb{T}^2$-manifold $M$; see the detailed proof of Theorem 5.1 in \cite{M} that $M$ admits a $\mathbb{T}^2$-invariant almost complex structure. Moreover, $M$ has $k$ fixed points, and the $\mathbb{T}^2$-weights at each fixed point $p_i=q_{i,1}$ are $\{v_i,-v_{i-1}\}$ for each $1 \leq i \leq k$. This proves the first part.

Suppose that $v_{i,1} \neq 0$ for all $i$. Then the second part of this theorem for the restriction of $\mathbb{T}^2$ on the first $S^1$-factor follow from the construction. From the construction, the $S^1$-weights at $p_i$ are $\{v_{i,1},-v_{i-1,1}\}$ for $1 \leq i \leq k$. Moreover, from the construction it follows that for any $1 \leq i \leq k$, $p_i$($=q_{i,1}$) and $p_{i+1}$($=q_{i,2}=q_{i+1,1}$) lie in the same 2-sphere $S^2=\mathbb{CP}^1=\{[z_1,z_2,0]:(z_1,z_2) \in \mathbb{C}^2-\{0\}\} \subset D_{a_i}(v_i,v_{i-1})$, where the first $S^1$-factor of $\mathbb{T}^2$ restricts to act on this $S^2$ by $t_1 \cdot [z_1,z_2,0]=[z_1, t_1^{v_{i,1}}z_2,0]$ for all $t_1 \in S^1$, with fixed points $p_i$ and $p_{i+1}$ that have $S^1$-weights $v_{i,1}$ and $-v_{i,1}$, respectively. Therefore, if we draw a 2-regular labeled directed multigraph $\Gamma$ so that $(p_i,p_{i+1})$ is $v_{i,1}$-edge for $1 \leq i \leq k$, then $\Gamma$ describes $M$ for the action of the first $S^1$-factor of $\mathbb{T}^2$.

Finally, suppose that the winding number of $v_0,v_1,\cdots,v_{k}$ is one. Then $v_i$ form a fan and thus the manifold $M$ is a complete non-singular surface. This implies that $M$ is projective, and hence K\"ahler.
\end{proof}

\begin{rem}
There is a slight difference between constructions of $\mathbb{T}^2$-manifolds in Theorem \ref{t31} of this paper and Theorem 5.1 of \cite{M}. In Theorem 5.1 of \cite{M}, Masuda also took $v_i$ in counterclockwise order but then used the dual basis $\{u_1^{(i)},u_2^{(i)}\}$ of $\{v_{i-1},v_i\}$ to construct a $\mathbb{T}^2$-manifold. Let $\{u_1^{(i)},u_2^{(i)}\}$ be the dual basis of $\{v_{i-1},v_i\}$. Then Masuda used $D_{a_i}(u_1^{(i)},u_2^{(i)})$ to construct a manifold, where the $\mathbb{T}^2$-weights at $q_{i,1}$ and $q_{i,2}$ are $\{u_1^{(i)},u_2^{(i)}\}$ and $\{-u_1^{(i)},u_2^{(i)}-a_iu_1^{(i)}\}$, respectively. For $1 \leq i \leq k$, Condition (2) of Theorem \ref{t31} can be written as
\begin{center}
$\begin{pmatrix} v_{i} & v_{i+1} \end{pmatrix} = \begin{pmatrix} v_{i-1} & v_{i} \end{pmatrix} \begin{pmatrix} 0 & -1 \\ 1 & -a_i \end{pmatrix}$
\end{center}
Since $\{u_1^{(i)},u_2^{(i)}\}$ is the dual basis of $\{v_{i-1},v_i\}$, for each $1 \leq i \leq k$, the relationship between $\{u_1^{(i+1)}, u_2^{(i+1)}\}$ and $\{u_1^{(i)}, u_2^{(i)}\}$ are then
\begin{equation} \label{e1}
\begin{pmatrix} u_1^{(i+1)} \\ u_2^{(i+1)} \end{pmatrix} = \begin{pmatrix} -a_i & 1 \\ -1 & 0 \end{pmatrix} \begin{pmatrix} u_1^{(i)} \\ u_2^{(i)} \end{pmatrix}
\end{equation}
Since Equation \eqref{e1} holds, the $\mathbb{T}^2$-weights at $q_{i,2}$ and $q_{i+1,1}$ agree; hence we can equivariantly plumb two manifolds $D_{a_i}(u_1^{(i)},u_2^{(i)})$ and $D_{a_{i+1}}(u_1^{(i+1)},u_2^{(i+1)})$ at $q_{i,2}$ and $q_{i+1,1}$ for each $1 \leq i \leq k$.

On the other hand, instead of taking the dual basis of $\{v_i,v_{i-1}\}$ if we take $u_1^{(i)}=v_i$ and $u_2^{(i)}=-v_{i-1}$, then Equation \eqref{e1} also holds. Therefore, in Theorem \ref{t31} we directly used $D_{a_i}(v_i,-v_{i-1})$ to construct a $\mathbb{T}^2$-manifold without introducing $u_m^{(i)}$'s. Since the $\mathbb{T}^2$-weights at $q_{i,2}$ and $q_{i+1,1}$ also agree in our case, we can perform equivariant plumbing. Therefore, the proofs for the arguments in Theorem 5.1 of \cite{M} apply smoothly for the rest of Theorem \ref{t31} (such as, by pasting $N$ and $D^2 \times \mathbb{T}^2$ along the boundary $\partial N \simeq S^1 \times \mathbb{T}^2$ of $N$, we get a 4-dimensional closed connected almost complex $\mathbb{T}^2$-manifold $M$, and $M$ admits a $\mathbb{T}^2$-invariant almost complex structure, etc).
\end{rem}

\begin{rem} \label{r34}
Theorem \ref{t11} does not imply that if a multigraph describes a 4-dimensional almost complex $S^1$-manifold, the manifold is unique. That is, a multigraph can describe two different $S^1$-manifolds. For this, let $M$ be an $S^1$-manifold with 8 fixed points constructed from Theorem \ref{t31} by taking $v_1=(1,0)$, $v_2=(1,1)$, $v_3=(-1,0)$, $v_4=(-1,-1)$, $v_5=(1,0)$, $v_6=(1,1)$, $v_7=(-1,0)$, $v_8=(-1,-1)$.
Next, for $i=1,2$, let $M_i$ be a $\mathbb{T}^2$-manifold with 4 fixed points constructed from Theorem \ref{t31} by taking $v_1=(1,0)$, $v_2=(1,1)$, $v_3=(-1,0)$, $v_4=(-1,-1)$.
By using Lemma \ref{l27}, take the equivariant connected sum of $M_1$ and $M_2$ along free orbits to construct a $\mathbb{T}^2$-manifold $M'$. From the constructions, $M$ and $M'$ are not diffeomorphic; $M$ has one chain of eight 2-spheres connecting 8 fixed points, while $M'$ has two chains of four 2-spheres, each chain connecting 4 fixed points. Both of $M$ and $M'$ have 2 fixed points having weights $\{1,1\}$, 4 fixed points having weights $\{-1,1\}$, and 2 fixed points having weights $\{-1,-1\}$, and hence they have the same fixed point data. Therefore, Figure \ref{fig1} describes both of $M$ and $M'$. On the other hand, a disjoint union of two copies of Figure \ref{fig6} also describes both of $M$ and $M'$. However, from the constructions, it is natural to take Figure \ref{fig1} to describe $M$, and to take a disjoint union $\Gamma'$ of two copies of Figure \ref{fig6} to describe $M'$.
\end{rem}

\section{Extendability of a graph} \label{extends}

In this section, we introduce the concept of extendability of a 2-regular labeled directed multigraph and prove several properties in terms of this extendability. We explain the content of this section. Definition \ref{extend} introduces the concept of extendability of a 2-regular labeled directed multigraph to $\mathbb{Z}^2$. Proposition \ref{p37} states that if such a multigraph  extends to $\mathbb{Z}^2$, then the multigraph describes some 4-dimensional almost complex $\mathbb{T}^2$-manifold for the action of the first $S^1$-factor. From Lemma \ref{graphequal} we know that if the circle group acts effectively on a 4-dimensional compact almost complex manifold $M$ with a discrete fixed point set, then there is a multigraph $\Gamma$ describing $M$ that is admissible and effective and satisfies the equal modulo property. Theorem \ref{Textend} asserts that the converse holds. Moreover, Theorem \ref{Textend} shows that the multigraph $\Gamma$ is obtained from a minimal multigraph (an admissible semi-free multigraph) followed by the 4 operations in Theorem \ref{t11}. The proof of Theorem \ref{Textend} uses several lemmas; Lemma \ref{l32} shows that a minimal multigraph extends to $\mathbb{Z}^2$, and Lemmas \ref{l33}, \ref{l34}, \ref{l35}, and \ref{l36} show that if $\Gamma$ extends to $\mathbb{Z}^2$ and $\Gamma'$ is obtained from $\Gamma$ by applying Operation 1, 2, 3, and 4 of Theorem \ref{t11}, respectively, then $\Gamma'$ also extends to $\mathbb{Z}^2$. With these, we will prove Theorems \ref{t01}, \ref{t314}, and \ref{t11} in Section \ref{s7}.

\begin{Definition} \label{extend}
We say that a connected 2-regular labeled directed multigraph $\Gamma$ \textbf{extends to} $\mathbb{Z}^2$ if we can label successive vertices of $\Gamma$ by $p_1,\cdots,p_k$ and choose a sequence of vectors $v_0,v_1,\cdots,v_k,v_{k+1}$ in $\mathbb{Z}^2$ in counterclockwise order with $v_k=v_0$ and $v_{k+1}=v_1$, so that the following hold.
\begin{enumerate}[(1)]
\item Each successive pair $v_i$ and $v_{i+1}$ is a basis of $\mathbb{Z}^2$, for $1 \leq i \leq k$.
\item For $1 \leq i \leq k$, there exists an integer $a_i$ such that $v_{i+1}=-a_i v_i-v_{i-1}$.
\item For $1 \leq i \leq k$, $(p_i,p_{i+1})$ is $v_{i,1}$-edge, where $v_i=(v_{i,1},v_{i,2}) \in \mathbb{Z}^2$. Here, $p_{k+1}=p_1$.
\end{enumerate}
\end{Definition}

\begin{Definition}
We say that a 2-regular labeled directed multigraph $\Gamma$ \textbf{extends to} $\mathbb{Z}^2$ if each component of $\Gamma$ extends to $\mathbb{Z}^2$.
\end{Definition}

To prove Theorem \ref{t02} later, we will need that if $\Gamma$ with $T(\Gamma)=1$ is connected and extends to $\mathbb{Z}^2$, then a manifold that we construct in Theorem \ref{t31} is K\"ahler. For this, we need the following simple lemma.

\begin{lemma} \label{winding}
Let $\Gamma$ be a connected 2-regular labeled directed multigraph that extends to $\mathbb{Z}^2$. Then the number $T(\Gamma)$ of vertices of index 0 is equal to the winding number of $v_i$, where $v_i$ are as in Definition \ref{extend}.
\end{lemma}

\begin{proof}
The winding number of $v_i$ is equal to the number of $j \in \{1,2,\cdots,k\}$ such that $v_j$ lies in the second or third quadrant and $v_{j+1}$ lies in the fourth or first quadrant, equivalently, $v_{j,1}$ is negative and $v_{j+1,1}$ is positive, where $v_j=(v_{j,1},v_{j,2})$. This is equivalent to saying that $(p_j,p_{j+1})$ is $v_{j,1}$-edge and $(p_{j+1},p_{j+2})$ is $v_{j+1,1}$-edge with $v_{j,1}<0$ and $v_{j+1,1}>0$, that is, $p_{j+1}$ has index 0.
\end{proof}

If $\Gamma$ extends to $\mathbb{Z}^2$, then $\Gamma$ describes some 4-dimensional almost complex $S^1$-manifold.

\begin{pro} \label{p37}
Let $\Gamma$ be a 2-regular labeled directed multigraph. Suppose that $\Gamma$ extends to $\mathbb{Z}^2$. Then there exists a 4-dimensional compact, connected almost complex manifold $M$ equipped with an effective $\mathbb{T}^2$-action having a discrete fixed point set, such that $\Gamma$ describes $M$ for the action of the first $S^1$-factor of $\mathbb{T}^2$.

If in addition $T(\Gamma)=1$, then there exists such a manifold $M$ that admits a K\"ahler structure, preserved by the $\mathbb{T}^2$-action.
\end{pro}

\begin{proof}
Let $\Gamma_1,\cdots,\Gamma_m$ be components of $\Gamma$. Since $\Gamma_j$ extends to $\mathbb{Z}^2$, we can label successive vertices of $\Gamma_j$ by $p_1,\cdots,p_k$ and choose a sequence of vectors $v_0,v_1,\cdots,v_k,v_{k+1}$ in $\mathbb{Z}^2$ in counterclockwise order with $v_k=v_0$ and $v_{k+1}=v_1$, so that the following hold.
\begin{enumerate}[(1)]
\item Each successive pair $v_i$ and $v_{i+1}$ is a basis of $\mathbb{Z}^2$, for $1 \leq i \leq k$.
\item For $1 \leq i \leq k$, there exists an integer $a_i$ such that $v_{i+1}=-a_i v_i-v_{i-1}$.
\item For $1 \leq i \leq k$, $(p_i,p_{i+1})$ is $v_{i,1}$-edge, where $v_i=(v_{i,1},v_{i,2}) \in \mathbb{Z}^2$. Here, $p_{k+1}=p_1$.
\end{enumerate}
By (2) of Theorem \ref{t31}, there exists a 4-dimensional compact connected almost complex manifold $M_j$ equipped with an effective $\mathbb{T}^2$-action such that $\Gamma_j$ describes $M_j$ for the action of the first $S^1$-factor. Let $M$ be an equivariant connected sum along free orbits of $M_1,\cdots,M_m$. By Lemma \ref{l27}, $\Gamma$ describes $M$ for the action of the first $S^1$-factor of $\mathbb{T}^2$. This proves the first claim.

Suppose that $\Gamma$ has exactly one vertex of index 0. By Definition \ref{extend} and Lemma \ref{winding}, $\Gamma$ is connected and the winding number of $v_0,v_1,\cdots,v_{k}$ is one. Thus, the last claim of Theorem \ref{t31} implies the second claim of this proposition.
\end{proof}

A minimal multigraph of Theorem \ref{t11} extends to $\mathbb{Z}^2$.

\begin{lem} \label{l32}
Any 2-regular admissible semi-free multigraph extends to $\mathbb{Z}^2$.
\end{lem}

\begin{proof}
Consider a 2-regular admissible semi-free multigraph. Let $\Gamma$ be a component. Since $\Gamma$ is admissible and semi-free, $\Gamma$ has $4k$ vertices; $k$ vertices of index 0, $2k$ vertices of index 1, and $k$ vertices of index 2. Let $p_1$ be a vertex of $\Gamma$ of index 0 and label successive vertices by $p_2, \cdots, p_{4k}$, see Figure \ref{fig9'}. Then $\Gamma$ extends to $\mathbb{Z}^2$ if we take the following $v_i$; $v_{4j+1}=(1,0)$, $v_{4j+2}=(1,1)$, $v_{4j+3}=(-1,0)$, $v_{4j+4}=(-1,-1)$ for each $0 \leq j \leq k-1$, and $a_i=0$ for all $1 \leq i \leq 4k$. \end{proof}

\begin{figure}
\centering
\begin{subfigure}[b][4.7cm][s]{1\textwidth}
\centering
\vfill
\begin{tikzpicture}[state/.style ={circle, draw}]
\node[state] (A) {$p_{4j+1}$};
\node[state] (B) [above right=of A] {$p_{4j+2}$};
\node[state] (C) [above right=of B] {$p_{4j+3}$};
\node[state] (D) [below right=of C] {$p_{4j+4}$};
\node[state] (E) [below right=of D] {$p_{4j+5}$};
\node[state] (F) [above right=of E] {$p_{4j+6}$};
\path (A) [->] edge node[left] {$1, (1,0)$} (B);
\path (B) [->] edge node [left] {$1$, $v_{4j+2}=(1,1)$} (C);
\path (D) [->] edge node [right] {$1$, $v_{4j+3}=(-1,0)$} (C);
\path (E) [->] edge node [left] {$1$, $v_{4j+4}=(-1,-1)$} (D);
\path (E) [->] edge node [right] {$1, (1,0)$} (F);
\end{tikzpicture}
\vspace{\baselineskip}
\end{subfigure}\qquad
\caption{Construction in Lemma \ref{l32}; label of each edge and value of $v_i$ at each $p_i$}\label{fig9'}
\end{figure}

The following lemma corresponds to Operation 1 (Figure \ref{fig2}) in Theorem \ref{t11}.

\begin{lem} \label{l33}
Let $\Gamma$ be a connected 2-regular labeled directed multigraph. Suppose that in $\Gamma$, $(p_{i-1},p_i)$ is $a$-edge and $(p_i,p_{i+1})$ is $b$-edge for some vertices $p_{i-1}$, $p_i$, $p_{i+1}$ and for some positive integers $a$ and $b$ (Figure \ref{fig2-1}). Let $\Gamma'$ be a multigraph obtained from $\Gamma$ by replacing the vertex $p_i$ with $(a+b)$-edge $(p_i',p_i'')$ so that $(p_{i-1},p_i')$ is $a$-edge and $(p_i'',p_{i+1})$ is $b$-edge (Figure \ref{fig2-2}). If $\Gamma$ extends to $\mathbb{Z}^2$, then $\Gamma'$ extends to $\mathbb{Z}^2$.
\end{lem}

\begin{proof}
Since $\Gamma$ extends to $\mathbb{Z}^2$, we can label successive vertices of $\Gamma$ by $p_1,\cdots,p_k$ and choose a sequence of vectors $v_0,v_1,\cdots,v_k,v_{k+1}$ in $\mathbb{Z}^2$ in counterclockwise order with $v_k=v_0$ and $v_{k+1}=v_1$, so that the following hold.
\begin{enumerate}
\item Each successive pair $v_i$ and $v_{i+1}$ is a basis of $\mathbb{Z}^2$, for $1 \leq i \leq k$.
\item For $1 \leq i \leq k$, there exists an integer $a_i$ such that $v_{i+1}=-a_i v_i-v_{i-1}$.
\item For $1 \leq i \leq k$, $(p_i,p_{i+1})$ is $v_{i,1}$-edge, where $v_i=(v_{i,1},v_{i,2}) \in \mathbb{Z}^2$. Here, $p_{k+1}=p_1$.
\end{enumerate}
In particular, $v_{i-1,1}=a$ and $v_{i,1}=b$ since $(p_{i-1},p_i)$ is $a$-edge and $(p_i,p_{i+1})$ is $b$-edge.

Define $v_j'$ as follows.
\begin{center}
$v_1'=v_1$, $\cdots$, $v_{i-2}'=v_{i-2}$, $v_{i-1}'=v_{i-1}$, $v_i'=v_{i-1}+v_i$, $v_{i+1}'=v_i$, $\cdots$, $v_{k+1}'=v_{k}$, $v_{k+2}'=v_1'$.
\end{center}

Since $v_j$ satisfy Condition (1) of Definition \ref{extend}, so do $v_j'$. The $v_j'$ satisfy Condition (2) of Definition \ref{extend} with the following values of $a_j'$.
\begin{center}
$a_1'=a_1$, $\cdots$, $a_{i-2}'=a_{i-2}$, $a_{i-1}'=a_{i-1}-1$, $a_i'=-1$, $a_{i+1}'=a_i-1$, $a_{i+2}'=a_{i+1}$, $\cdots$, $a_{k+1}'=a_k$.
\end{center}
For instance, since $v_{i-1}'=v_{i-1}$, $v_{i-2}'=v_{i-2}$, and $a_{i-1}'=a_{i-1}-1$, it follows that $v_i'=v_{i-1}+v_i=v_{i-1}-a_{i-1}v_{i-1}-v_{i-2}=-(a_{i-1}-1) v_{i-1}-v_{i-2}=-a_{i-1}' v_{i-1}'-v_{i-2}'$.
The $v_j'$ also satisfy Condition (3) of Definition \ref{extend}; for instance, $(p_i', p_{i+1}')$ is $(a+b)$-edge in $\Gamma'$ and $v_{i,1}'=v_{i-1,1}+v_{i,1}=a+b$ because $v_{i-1,1}=a$ and $v_{i,1}=b$. Therefore, $\Gamma'$ extends to $\mathbb{Z}^2$. \end{proof}

The following lemma corresponds to Operation 2 (Figure \ref{fig3}) in Theorem \ref{t11}. 

\begin{lem} \label{l34}
Let $\Gamma$ be a connected 2-regular labeled directed multigraph. Suppose that $(p_1,p_2)$ is $(-c)$-edge, $(p_2,p_3)$ is $(-d)$-edge, and $(p_3,p_4)$ is $c$-edge, for some vertices $p_1$, $p_2$, $p_3$, $p_4$ and for some positive integers $c$, $d$ (Figure \ref{fig3-1}). Let $\Gamma'$ be a multigraph obtained from $\Gamma$  by replacing the label $d$ of the edge $(p_3,p_2)$ with $d+c$ (Figure \ref{fig3-2}). If $\Gamma$ extends to $\mathbb{Z}^2$, then $\Gamma'$ extends to $\mathbb{Z}^2$.
\end{lem}

\begin{proof}
Since $\Gamma$ extends to $\mathbb{Z}^2$, we can label successive vertices of $\Gamma$ by $p_1,\cdots,p_k$ and choose a sequence of vectors $v_0,v_1,\cdots,v_k,v_{k+1}$ in $\mathbb{Z}^2$ in counterclockwise order with $v_k=v_0$ and $v_{k+1}=v_1$, so that the following hold.
\begin{enumerate}
\item Each successive pair $v_i$ and $v_{i+1}$ is a basis of $\mathbb{Z}^2$, for $1 \leq i \leq k$.
\item For $1 \leq i \leq k$, there exists an integer $a_i$ such that $v_{i+1}=-a_i v_i-v_{i-1}$.
\item For $1 \leq i \leq k$, $(p_i,p_{i+1})$ is $v_{i,1}$-edge, where $v_i=(v_{i,1},v_{i,2}) \in \mathbb{Z}^2$. Here, $p_{k+1}=p_1$.
\end{enumerate}
In particular, $v_{1,1}=-c$, $v_{2,1}=-d$, and $v_{3,1}=c$ since $(p_1,p_2)$ is $(-c)$-edge, $(p_2,p_3)$ is $(-d)$-edge, and $(p_3,p_4)$ is $d$-edge. Because $v_3=-a_2 v_2-v_1$, it follows that $-c=v_{3,1}=-a_2 v_{2,1}-v_{1,1}=-a_2 (-d) -c$ and thus $a_2=0$ and $v_3=-v_1$.

Define $v_i'$ as follows.
\begin{center}
$v_1'=v_1$, $v_2'=v_1+v_2$, $v_3'=v_3$, $\cdots$, $v_{k+1}'=v_1'$.
\end{center}

Since $v_i$ satisfy Condition (1) of Definition \ref{extend}, so do $v_i'$; for example, $v_2'$ and $v_3'$ form a basis of $\mathbb{Z}^2$ because $v_2'=v_1+v_2$ and $v_3'=-v_1$, and $v_1$ and $v_2$ form a basis of $\mathbb{Z}^2$. The $v_i'$ satisfy Condition (2) of Definition \ref{extend} with the following values of $a_i'$.
\begin{center}
$a_1'=a_1-1$, $a_2'=0$, $a_3'=a_1+1$, $a_4'=a_4$, $\cdots$, $a_k'=a_k$, $a_{k+1}'=a_1-1$.
\end{center}
The $v_i'$ also satisfy Condition (3) of Definition \ref{extend}; for instance, $(p_2,p_3)$ is $(-c-d)$-edge in $\Gamma'$ and $v_{2,1}'=v_{1,1}+v_{2,1}=-c-d$ because $v_{1,1}=-c$ and $v_{2,1}=-d$. Therefore, $\Gamma'$ extends to $\mathbb{Z}^2$.
\end{proof}

The following lemma corresponds to Operation 3 (Figure \ref{fig4}) in Theorem \ref{t11}. 

\begin{lem} \label{l35} 
Let $\Gamma$ be a connected 2-regular labeled directed multigraph. Suppose that $(p_1,p_2)$ is $e$-edge, $(p_2,p_3)$ is $f$-edge, and $(p_3,p_4)$ is $(-e)$-edge, for some vertices $p_1$, $p_2$, $p_3$, $p_4$ and for some positive integers $e$, $f$ (Figure \ref{fig4-1}). Let $\Gamma'$ be a multigraph obtained from $\Gamma$ by replacing the label $f$ of the edge $(p_2,p_3)$ with $f+e$ (Figure \ref{fig4-2}). If $\Gamma$ extends to $\mathbb{Z}^2$, then $\Gamma'$ extends to $\mathbb{Z}^2$. \end{lem}

\begin{proof}
Exactly the same proof of Lemma \ref{l34} applies if we replace $c$ with $-e$ and $d$ with $-f$.
\end{proof}

The below lemma corresponds to Operation 4 (Figure \ref{fig5}) in Theorem \ref{t11}.

\begin{lem} \label{l36}
Let $\Gamma$ be a connected 2-regular labeled directed multigraph. Suppose that $(p_2,p_1)$ is $g$-edge, $(p_2,p_3)$ is $h$-edge, $(p_3,p_4)$ is $g$-edge, and $(p_5,p_4)$ is $h$-edge, for some vertices $p_1$, $p_2$, $p_3$, $p_4$, $p_5$ and for some positive integers $g$, $h$ (Figure \ref{fig5-1}). Let $\Gamma'$ be a multigraph obtained from $\Gamma$ by replacing $(p_2,p_3)$-edge and $(p_3,p_4)$-edge with $(p_2,p_4)$-edge (which is $(p_2',p_3')$-edge) with label $(g+h)$, removing $p_3$ (Figure \ref{fig5-2}). If $\Gamma$ extends to $\mathbb{Z}^2$, then $\Gamma'$ extends to $\mathbb{Z}^2$.
\end{lem}

\begin{proof}
Since $\Gamma$ extends to $\mathbb{Z}^2$, we can label successive vertices of $\Gamma$ by $p_1,\cdots,p_k$ and choose a sequence of vectors $v_0,v_1,\cdots,v_k,v_{k+1}$ in $\mathbb{Z}^2$ in counterclockwise order with $v_k=v_0$ and $v_{k+1}=v_1$, so that the following hold.
\begin{enumerate}
\item Each successive pair $v_i$ and $v_{i+1}$ is a basis of $\mathbb{Z}^2$, for $1 \leq i \leq k$.
\item For $1 \leq i \leq k$, there exists an integer $a_i$ such that $v_{i+1}=-a_i v_i-v_{i-1}$.
\item For $1 \leq i \leq k$, $(p_i,p_{i+1})$ is $v_{i,1}$-edge, where $v_i=(v_{i,1},v_{i,2}) \in \mathbb{Z}^2$. Here, $p_{k+1}=p_1$.
\end{enumerate}

Since $v_{1,1}=-g$, $v_{3,1}=g$, and $v_{3,1}=-a_2 v_{2,1}-v_{1,1}$, $a_1=0$ and $v_3=-v_1$. Similarly, since $v_{2,1}=h$, $v_{4,1}=-h$, and $v_{4,1}=-a_3 v_{3,1}-v_{2,1}$, $a_3=0$ and $v_4=-v_2$.

Define $v_i'$ as follows.
\begin{center}
$v_1'=v_1$, $v_2'=v_2+v_3$, $v_3'=v_4$, $\cdots$, $v_{k-1}'=v_k$, $v_k'=v_1'$.
\end{center}
Since $v_3=-v_1$, $v_4=-v_2$ and $v_1,v_2,v_3,v_4$ are in counterclockwise order, it follows that $v_1'=v_1$, $v_2'=v_2+v_3=v_2-v_1$, $v_3'=v_4=-v_2$ are also in counterclockwise order and $v_j'$ satisfy Condition (1) of Definition \ref{extend}. The $v_i'$ satisfy Condition (2) of Definition \ref{extend} with the following values of $a_i'$.
\begin{center}
$a_1'=a_1+1$, $a_2'=1$, $a_3'=a_4+1$, $a_4'=a_5$, $\cdots$, $a_{k-1}'=a_k$, $a_{k}'=a_1+1$.
\end{center}
The $v_i'$ also satisfy Condition (3) of Definition \ref{extend}; for instance, $(p_2',p_3')$ is $(h+g)$-edge in $\Gamma'$, and $v_{2,1}'=v_{2,1}+v_{3,1}=h+g$ since $v_{2,1}=h$ and $v_{3,1}=g$. Therefore, $\Gamma'$ extends to $\mathbb{Z}^2$. \end{proof}

\begin{rem} Note that for Lemma \ref{l33}, Lemma \ref{l34}, Lemma \ref{l35}, and Lemma \ref{l36}, the converse also holds. That is, if $\Gamma'$ extends to $\mathbb{Z}^2$, then so does $\Gamma$. \end{rem}

We can summarize the above lemmas as follows.

\begin{lemma} \label{extendcombine}
 Let $\Gamma$ be a multigraph obtained from a 2-regular admissible semi-free multigraph followed by a combination of the 4 operations in Theorem \ref{t11}. Then $\Gamma$ extends to $\mathbb{Z}^2$. 
\end{lemma}

\begin{proof}
By Lemma \ref{l32}, a 2-regular admissible semi-free multigraph extends to $\mathbb{Z}^2$. Therefore, by Lemmas \ref{l33}, \ref{l34}, \ref{l35}, and \ref{l36}, $\Gamma$ extends to $\mathbb{Z}^2$.
\end{proof}

For a 2-regular directed labeled multigraph, we provide a sufficient condition for the multigraph to extend to $\mathbb{Z}^2$ and to be obtained from a minimal multigraph followed by the 4 operations in Theorem \ref{t11}.

\begin{theo} \label{Textend} 
Let $\Gamma$ be a 2-regular admissible effective directed labeled multigraph that satisfies the equal modulo property. Then the following hold.
\begin{enumerate}[(1)]
\item The multigraph $\Gamma$ extends to $\mathbb{Z}^2$.
\item The multigraph $\Gamma$ is obtained from a 2-regular admissible semi-free multigraph followed by a combination of the 4 operations in Theorem \ref{t11}.
\end{enumerate}
\end{theo}

\begin{proof}
The proof is based on the induction on the largest label among the labels of all the edges of $\Gamma$. Let $l$ be the largest label among the labels of all the edges of $\Gamma$. Suppose that $l=1$. Then $\Gamma$ is a 2-regular admissible semi-free multigraph, and this theorem follows from Lemma \ref{l32}. 

Therefore, from now on, suppose that $l>1$. Suppose that $(p,q)$ is $l$-edge for some vertices $p$ and $q$. This means that $p$ has the weight $+l$ and $q$ has the weight $-l$. Then exactly one of the following holds for the indices of $p$ and $q$.
\begin{enumerate}
\item[(a)] $n_p=1$ and $n_q=1$.
\item[(b)] $n_p=0$ and $n_q=1$.
\item[(c)] $n_p=1$ and $n_q=2$.
\item[(d)] $n_p=0$ and $n_q=2$.
\end{enumerate}

Suppose that Case (a) holds. Then the weights at $p$ and $q$ are $\{-x_1,l\}$ and $\{-l,x_2\}$ respectively, for some positive integers $x_1$ and $x_2$. Since $\Gamma$ is effective and $l$ is the largest weight, $x_1<l$ and $x_2<l$. By the equal modulo property, the weights at $p$ and the weights at $q$ are equal modulo $l$, and this implies that $-x_1 \equiv x_2 \mod l$, that is, $x_1+x_2=l$. This is precisely the case as in Figure \ref{fig2-2} with $p_i'=p$, $p_i''=q$, $a=x_1$, and $b=x_2$.

Let $\Gamma'$ be a multigraph obtained from $\Gamma$ by shrinking the $l$-edge $(p,q)$ to a vertex (Figure \ref{fig2-1} with $a=x_1$ and $b=x_2$). Assume that $\Gamma'$ extends to $\mathbb{Z}^2$. Then by Lemma \ref{l33}, $\Gamma$ extends to $\mathbb{Z}^2$. Therefore, in Case (a), extendability of $\Gamma$ reduces to extendability of $\Gamma'$.
For (2), this step corresponds to Operation 1 in Theorem \ref{t11}.

Suppose that Case (b) holds. Then the weights at $p$ and $q$ are $\{x_1,l\}$ and $\{-l,x_2\}$ respectively, for some positive integers $x_1$ and $x_2$. Since $\Gamma$ is effective and $l$ is the largest weight, $x_1<l$ and $x_2<l$. By the equal modulo property, the weights at $p$ and the weights at $q$ are equal modulo $l$. This implies that $x_1 \equiv x_2 \mod l$, that is, $x_1=x_2$. This is precisely the case as in Figure \ref{fig3-2} with $p_{i+1}=p$, $p_i=q$, $c+d=l$, and $c=x_1$.

Let $\Gamma'$ be a multigraph obtained from $\Gamma$ by changing the label $l$ of the edge $(p,q)$ by $l-x_1$  (Figure \ref{fig3-1} with $p_{i+1}=p$, $p_i=q$, $c=x_1$, and $d=l-x_1$). Assume that $\Gamma'$ extends to $\mathbb{Z}^2$. Then by Lemma \ref{l34}, $\Gamma$ extends to $\mathbb{Z}^2$. Therefore, in Case (b), extendability of $\Gamma$ reduces to extendability of $\Gamma'$. For (2), this step corresponds to Operation 2 in Theorem \ref{t11}.

Suppose that Case (c) holds. Then the weights at $p$ and $q$ are $\{-x_1,l\}$ and $\{-l,-x_2\}$ respectively, for some positive integers $x_1$ and $x_2$. Since $\Gamma$ is effective and $l$ is the largest weight, $x_1<l$ and $x_2<l$. By the equal modulo property, the weights at $p$ and the weights at $q$ are equal modulo $l$, and this implies that $-x_1 \equiv -x_2 \mod l$, that is, $x_1=x_2$. This is precisely the case as in Figure \ref{fig4-2} with $p_i=p$, $p_{i+1}=q$, $e=x_1$, and $f+e=l$.

Let $\Gamma'$ (Figure \ref{fig4-1} with $p_i=p$, $p_{i+1}=q$, $e=x_1$, and $f=l-x_1$) be a multigraph obtained from $\Gamma$ by changing the label $l$ of the edge $(p,q)$ by $l-x_1$. Assume that $\Gamma'$ extends to $\mathbb{Z}^2$. Then by Lemma \ref{l35}, $\Gamma$ extends to $\mathbb{Z}^2$. Therefore, in Case (c), extendability of $\Gamma$ reduces to extendability of $\Gamma'$. For (2), this step corresponds to Operation 3 in Theorem \ref{t11}.

Suppose that Case (d) holds. Then the weights at $p$ and $q$ are $\{x_1,l\}$ and $\{-l,-x_2\}$ respectively, for some positive integers $x_1$ and $x_2$. Since $\Gamma$ is effective and $l$ is the largest weight, $x_1<l$ and $x_2<l$. By the equal modulo property, the weights at $p$ and the weights at $q$ are equal modulo $l$, and this implies that $x_1 \equiv -x_2 \mod l$, that is, $x_1+x_2=l$. This is precisely the case as in Figure \ref{fig5-2}. Let $p'$ be the terminal point of the $x_1$-edge whose initial point is $p$, and let $q'$ be the initial point of the $x_2$-edge whose terminal point is $q$, that is, $(p,p')$ is $x_1$-edge and $(q',q)$ is $x_2$-edge. In other words, $p',p,q,q'$ correspond to $p_{i-2},p_{i-1},p_{i+1},p_{i+2}$ in Figure \ref{fig5-2} with $g=x_1$ and $h=x_2$. 

Let $\Gamma'$ be a multigraph obtained from $\Gamma$ by replacing the edge $(p,q)$ with two edges $(p,r)$ and $(r,q)$ with labels $x_2$ and $x_1$ respectively with adding a vertex $r$ (Figure \ref{fig5-1} with $p_{i-1}=p$, $p_{i+1}=q$, $p_i=r$, $g=x_1$, and $h=x_2$). Assume that $\Gamma'$ extends to $\mathbb{Z}^2$. Then by Lemma \ref{l36}, $\Gamma$ extends to $\mathbb{Z}^2$. Therefore, in Case (d), extendability of $\Gamma$ reduces to extendability of $\Gamma'$. For (2), this step corresponds to Operation 4 in Theorem \ref{t11}.

Whenever $(p,q)$ is $l$-edge, by the four steps above, extendability of $\Gamma$ reduces to extendability of $\Gamma'$ that has one less $l$-edge than $\Gamma$. The multigraph $\Gamma'$ is admissible and effective, and satisfies the equal modulo property. By repeating the arguments, extendability of $\Gamma$ reduces to extendability of a 2-regular labeled directed multigraph $\Gamma''$ that is admissible and has edges whose labels are all 1. That is, $\Gamma''$ is a 2-regular admissible semi-free multigraph. By Lemma \ref{l32}, $\Gamma''$ extends to $\mathbb{Z}^2$. Therefore, both of (1) and (2) of this theorem follow. \end{proof}

Proposition \ref{p37} and Theorem \ref{Textend} together imply the following.

\begin{theorem}
Let $\Gamma$ be a 2-regular admissible effective directed labeled multigraph that satisfies the equal modulo property. Then there exists a 4-dimensional compact, connected almost complex manifold $M$ equipped with an effective $\mathbb{T}^2$-action having a discrete fixed point set, such that $\Gamma$ describes $M$ for the action of the first $S^1$-factor of $\mathbb{T}^2$.
\end{theorem}

We see that the following two conditions are equivalent for a 2-regular directed labeled multigraph.

\begin{lemma}
Let $\Gamma$ be a 2-regular directed labeled multigraph. Then the following are equivalent.
\begin{enumerate}
\item The multigraph $\Gamma$ is admissible and effective and satisfies the equal modulo property.
\item The multigraph $\Gamma$ is obtained from an admissible semi-free multigraph followed by a combination of the 4 operations in Theorem \ref{t11}.
\end{enumerate}
\end{lemma}

\begin{proof}
Suppose that $\Gamma$ is admissible and effective and satisfies the equal modulo property. By Theorem \ref{Textend}, $\Gamma$ is obtained from an admissible semi-free multigraph followed by a combination of the 4 operations in Theorem \ref{t11}.

Any 2-regular admissible semi-free multigraph is effective and satisfies the equal modulo property. Let $\Gamma'$ be a multigraph obtained from a 2-regular admissible semi-free multigraph followed by a combination of the 4 operations in Theorem \ref{t11}. Suppose that $\Gamma'$ is admissible and effective and satisfies the equal modulo property. Let $\Gamma$ be a multigraph obtained from $\Gamma'$ by applying one of the 4 operations in Theorem \ref{t11}. One can check that $\Gamma$ is also admissible and effective and satisfies the equal modulo property; see Figures \ref{fig2}, \ref{fig3}, \ref{fig4}, and \ref{fig5}.
\end{proof}

\section{Proofs of Theorems \ref{t01}, \ref{t314}, and \ref{t11}} \label{s7}

In this section, with the preliminaries established in the previous sections, we prove Theorems \ref{t01}, \ref{t314}, and \ref{t11}. First, we prove Theorem \ref{t01}, which classifies the weights at the fixed points of a 4-dimensional compact almost complex $S^1$-manifold with a discrete fixed point set. To simplify its proof, we introduce some notions.

\begin{Definition}
Let $S=\{p_1,\cdots,p_k\}$ be a set, where each $p_i=\{a_i,b_i\}$ is an unordered pair of integers. We call $p_i$ a \textbf{point} and $a_i$ and $b_i$ the \textbf{weights} of $p_i$.
The \textbf{index} of $p_i$ is the number of negative elements of $a_i$ and $b_i$.
\end{Definition}

\begin{Definition} \label{generated}
Let $S=\{p_1,\cdots,p_k\}$ be a set of unordered pairs $p_i=\{a_i,b_i\}$ of relatively prime integers with the following properties.
\begin{enumerate}[(1)]
\item The number of weights that are $+1$ at points of index $j$ is equal to the number of weights that are $-1$ at points of index $j+1$ for $j \in \{0,1\}$.
\item Given an integer $w > 1$ and a relatively prime integer $x$, the following sets have the same cardinality.
\begin{enumerate}[(a)]
\item Points with one weight $+w$ and the other equal to $x$ modulo $w$.
\item Points with one weight $-w$ and the other equal to $x$ modulo $w$.
\end{enumerate}
\end{enumerate}
A \textbf{multigraph generated by $\textbf{S}$} is a 2-regular directed labeled multigraph, constructed as follows.
\begin{enumerate}[(i)]
\item To each $p_i$ we assign a vertex, also denoted $p_i$.
\item From (1), each time some point $p_i$ has the weight $+1$ and index $m$, there is a point $p_j$ that has the weight $-1$ and index $m+1$; then we draw an edge from $p_i$ to $p_j$ with label $1$. This exhausts all weights $+1$ and $-1$ that occur over all points of $S$.
\item Let $w>1$ be an integer. By (2), each time some point $p_i$ has the weight $w$, there is another point $p_j$ that has the weight $-w$, such that $x \equiv y \mod w$, where $x$ ($y$) is the other weight of $p_i$ ($p_j$) and is relatively prime to $w$; we then draw an edge from $p_i$ to $p_j$ with label $w$. Doing this for each positive integer $w>1$, we exhaust all weights over all points, to draw edges.
\end{enumerate}
\end{Definition}

With the above, we prove Theorem \ref{t01}.

\begin{proof}[\textbf{Proof of Theorem \ref{t01}}]
Let the circle act effectively on a 4-dimensional compact connected almost complex manifold $M$ with a discrete fixed point set. By Lemma \ref{graphequal}, there exists a labeled directed multigraph $\Gamma$ describing $M$ that is admissible and effective and satisfies the equal modulo property. Since $\Gamma$ is admissible, (1) of Theorem \ref{t01} holds. Since $\Gamma$ is effective and satisfies the equal modulo property, (2) of Theorem \ref{t01} holds. This proves one direction of Theorem \ref{t01}.

To prove the converse, let $S=\{p_1,\cdots,p_k\}$ be a set, where each $p_i=\{a_i,b_i\}$ is an unordered pair of relatively prime integers. Suppose that the following hold.
\begin{enumerate}[(1)]
\item The number of weights that are $+1$ at points of index $j$ is equal to the number of weights that are $-1$ at points of index $j+1$ for $j \in \{0,1\}$.
\item Given an integer $w > 1$ and a relatively prime integer $x$, the following sets have the same cardinality.
\begin{enumerate}[(a)]
\item Points with one weight $+w$ and the other equal to $x$ modulo $w$.
\item Points with one weight $-w$ and the other equal to $x$ modulo $w$.
\end{enumerate}
\end{enumerate}

Let $\Gamma$ be a multigraph generated by $S$. By construction, $\Gamma$ is 2-regular, directed, and labeled. By (ii) of Definition \ref{generated}, $\Gamma$ is admissible. By (iii) of Definition \ref{generated}, $\Gamma$ is effective and satisfies the equal modulo property. Therefore, by Theorem \ref{Textend}, $\Gamma$ extends to $\mathbb{Z}^2$. By Proposition \ref{p37}, there exists a 4-dimensional compact, connected almost complex manifold $M$ equipped with a $\mathbb{T}^2$-action having a discrete fixed point set, such that $\Gamma$ describes $M$ for the action of the first $S^1$-factor of $\mathbb{T}^2$. The last claim of this theorem follows from Lemma \ref{2sphere}.
\end{proof}

Next, we prove Theorem \ref{t314}; any set of pairs of integers that arises as the weights at the fixed points of a circle action on a 4-dimensional compact almost complex manifold with a discrete fixed point set, also arises as the weights of the restriction of a $\mathbb{T}^2$-action.

\begin{proof}[\textbf{Proof of Theorem \ref{t314}}]
By Lemma \ref{graphequal}, there exists a labeled directed multigraph $\Gamma$ describing $M$ that is admissible and effective and satisfies the equal modulo property. By Theorem \ref{Textend}, $\Gamma$ extends to $\mathbb{Z}^2$. By Proposition \ref{p37}, there exists a 4-dimensional compact, connected almost complex manifold $M'$ equipped with a $\mathbb{T}^2$-action having a discrete fixed point set, such that $\Gamma$ describes $M'$ for the action of the first $S^1$-factor of $\mathbb{T}^2$.
\end{proof}

We prove Theorem \ref{t11}, which provides a minimal multigraph and operations for a multigraph containing the fixed point data of a 4-dimensional compact almost complex $S^1$-manifold with a discrete fixed point set

\begin{proof}[\textbf{Proof of Theorem \ref{t11}}]
Let the circle act effectively on a 4-dimensional compact almost complex manifold $M$ with a discrete fixed point set. By Lemma \ref{graphequal}, there exists a labeled directed multigraph $\Gamma$ describing $M$ that is admissible and effective and satisfies the equal modulo property. By Theorem \ref{Textend}, $\Gamma$ is obtained from a 2-regular admissible semi-free multigraph followed by a combination of the 4 operations in Theorem \ref{t11}.

By Theorem \ref{t21}, the Todd genus of $M$ is equal to the number of fixed points of index 0, which is equal to the number $T(\Gamma)$ of vertices of $\Gamma$ of index 0.

We prove the last claim of Theorem \ref{t11}. Let $\Gamma$ be a multigraph obtained from a 2-regular admissible semi-free multigraph followed by a combination of the 4 operations in Theorem \ref{t11}. 
By Lemma \ref{extendcombine}, $\Gamma$ extends to $\mathbb{Z}^2$. Then the last claim follows from Proposition \ref{p37}.
\end{proof}

\section{Chern numbers} \label{s5}

In this section, we prove Theorem \ref{c115} and Theorem \ref{c19}. Theorem \ref{c115} provides a necessary and sufficient condition for the Chern numbers of a 4-dimensional compact almost complex $S^1$-manifold with a discrete fixed point set, and Theorem \ref{c19} provides a necessary and sufficient condition for the number $N_i$ of fixed points of index $i$ for such a manifold.

\begin{proof}[\textbf{Proof of Theorem \ref{c19}}]
Let the circle act on a 4-dimensional compact connected almost complex manifold $M$ with exactly $N_0$ fixed points of index $0$, $N_1$ fixed points of index $1$, and $N_2$ fixed points of index $2$, where not all $N_0$, $N_1$, $N_2$ are zero. By Theorem \ref{t21}, $N_0=N_2$ and the Todd genus of $M$ is equal to $N_0$. Corollary \ref{c23} then implies that $N_0$, $N_1$, $N_2$ are all positive. This proves one direction.

Fix positive integers $N_0$ and $N_1$. Let $k=2N_0+1$. Let $S_1$ be a set 
\begin{center}
$S_1=\{\{k,1\},\{-1,2\},\{-2,-3\},\{3,4\},\{-4,-5\},\{5,6\},\{-6,-7\},\{7,8\},\cdots,\{k-2,k-1\},\{-k-1,-k\}\}$. 
\end{center}
For $1 \leq i \leq N_1-1$, define inductively a set $S_{i+1}$ from $S_i$ as follows: the set $S_i$ has at least one point $p$ with index 1. Let $-a$ and $b$ be the weights at $p$ for some positive integers $a$ and $b$. Then replace $p=\{-a,b\}$ with two points $p_1=\{-a,a+b\}$ and $p_2=\{-a-b,b\}$. One can check that for each $1 \leq i \leq N_1-1$, since $S_i$ satisfies Conditions (1) and (2) of Theorem \ref{t01}, so does $S_{i+1}$; moreover, $S_{i+1}$ has one more point of index 1 than $S_i$ and hence $S_{i+1}$ has $N_0$ points of index 0, $i+1$ points of index 1, and $N_0$ points of index 2. Applying Theorem \ref{t01} to the set $S_{N_1}$, there exists an effective circle action on a 4-dimensional compact, connected almost complex manifold $M$ with $2N_0+N_1$ fixed points, $N_0$ points of index 0, $N_1$ points of index 1, and $N_0$ points of index 2.
\end{proof}

\begin{proof}[\textbf{Proof of Theorem \ref{c115}}] 
This theorem follows from Theorem \ref{c19} and the fact that $c_1^2[M]=10 N_0- N_1$ and $c_2[M]=2N_0+N_1$, where $N_i$ is the number of fixed points of index $i$; see \cite[$\S$6.1]{S}. \end{proof} 

\section{Comparison with complex manifolds and symplectic manifolds} \label{s4}

In this section, (1) we compare circle actions with discrete fixed point sets on 4-dimensional almost complex, complex, and symplectic manifolds, (2) prove Theorem \ref{t02}, (3) establish a version of Theorem \ref{t11} for a complex or symplectic manifold (Theorems \ref{t12} and \ref{t83}), and (4) compare the known classifications for 4-dimensional complex or symplectic $S^1$-manifolds and our classification for these manifolds (Theorem \ref{t12}) that is in terms of our minimal model (multigraph) and operations.

We compare circle actions on almost complex manifolds, complex manifolds, and symplectic manifolds in low dimensions. In dimension 2, there is no difference; if a 2-dimensional compact connected oriented manifold $M$ admits a circle action, either $M$ is the 2-sphere $S^2$ or the 2-torus $\mathbb{T}^2$, and both of them are K\"ahler; hence both of them are almost complex, complex, and symplectic.

In dimension 4, however, the class of almost complex $S^1$-manifolds is bigger than that of complex $S^1$-manifolds and that of symplectic $S^1$-manifolds. If the circle acts on a 4-dimensional compact connected almost complex manifold $M$ with a non-empty discrete fixed point set, the Todd genus of $M$ takes any positive integer. On the other hand, if in addition $M$ is complex or symplectic and the circle action preserves the given structure, the Todd genus of $M$ is 1. Therefore, there are almost complex manifolds equipped with circle actions which cannot be complex or symplectic.

After proving the following lemma, we prove Theorem \ref{t02}.

\begin{lem} \label{todd}
Let $M$ be a 4-dimensional compact connected manifold, complex or symplectic. Let the circle act on $M$ preserving the complex or symplectic structure. Suppose that the fixed point set is non-empty and discrete. Then the Todd genus of $M$ is 1.
\end{lem}

\begin{proof}
If $M$ is a complex manifold, the Todd genus of $M$ is equal to 1; see \cite{CHK}. Suppose that $M$ is a symplectic manifold. Since there is a fixed point and $\dim M=4$, the action is Hamiltonian; see \cite{MD}. Since the action is Hamiltonian, there is a unique fixed point of index 0. By Theorem \ref{t21}, the Todd genus of $M$ is equal to the number of fixed points of index 0. \end{proof}

Recall that Theorem \ref{t02} provides a necessary and sufficient condition for the weights at the fixed points of an $S^1$-action on a 4-dimensional compact connected manifold with a discrete fixed point set, which is complex or symplectic.

\begin{proof}[\textbf{Proof of Theorem \ref{t02}}]
By Lemma \ref{todd}, the Todd genus of $M$ is 1. Theorem \ref{t21} implies that $M$ has exactly one fixed point of index 0. Then one direction of Theorem \ref{t02} follows from Theorem \ref{t01}.

To prove the converse, for $1 \leq i \leq k$, let $p_i=\{a_i,b_i\}$ be an unordered pair of relatively prime non-zero integers with the following properties.
\begin{enumerate}[(1)]
\item The number of weights that are $+1$ at fixed points of index $j$ is equal to the number of weights that are $-1$ at fixed points of index $j+1$ for $j \in \{0,1\}$.
\item Given an integer $w > 1$ and a relatively prime integer $x$, the following sets have the same cardinality.
\begin{enumerate}[(a)]
\item Points $p_i$ with one weight $+w$ and the other equal to $x$ modulo $w$.
\item Points $p_i$ with one weight $-w$ and the other equal to $x$ modulo $w$.
\end{enumerate}
\item There is exactly one point whose weights are all positive.
\end{enumerate}

Let $S=\{p_1,\cdots,p_k\}$. Let $\Gamma$ be a multigraph generated by $S$ (Definition \ref{generated}). Then $\Gamma$ is a 2-regular admissible effective directed labeled multigraph satisfying the equal modulo property. By Theorem \ref{Textend}, $\Gamma$ extends to $\mathbb{Z}^2$. Since there exists exactly one $i$ such that both $a_i$ and $b_i$ are positive, $\Gamma$ has exactly one vertex of index 0. This implies that $\Gamma$ is connected. By Lemma \ref{winding}, the winding number of $v_i$ is 1, where $v_i$ are as in Definition \ref{extend}. Then by Proposition \ref{p37}, the converse holds.
\end{proof}

We show that if $M$ is complex or symplectic, a multigraph describing $M$ can be constructed from a single 2-regular admissible semi-free multigraph $\Gamma$ with $T(\Gamma)=1$ (Figure \ref{fig6}) followed by the 4 operations in Theorem \ref{t11}.

\begin{theo} \label{t12} Let $M$ be a 4-dimensional compact connected manifold, complex or symplectic. Let the circle act on $M$ preserving the complex or symplectic structure. Suppose that the fixed point set is non-empty and discrete. Then a multigraph describing $M$ can be constructed by beginning with a 2-regular admissible semi-free multigraph $\Gamma$ with $T(\Gamma)=1$ (Figure \ref{fig6}) followed by a combination of the 4 operations in Theorem \ref{t11}. 
\end{theo}

\begin{figure}
\centering
\begin{tikzpicture}[state/.style ={circle, draw}]
\node[state] (A) {$p_1$};
\node[state] (B) [above left=of A] {$p_4$};
\node[state] (C) [above right=of A] {$p_2$};
\node[state] (D) [above right=of B] {$p_3$};
\path (A) [->] edge node[right] {$1$} (B);
\path (A) [->] edge node [right] {$1$} (C);
\path (B) [->] edge node [right] {$1$} (D);
\path (C) [->] edge node [right] {$1$} (D);
\end{tikzpicture}
\vfill
\caption{Minimal multigraph for $M$ complex or symplectic}\label{fig6}
\end{figure}

\begin{proof}
We apply Theorem \ref{t11}. By Theorem \ref{t11}, a labeled directed multigraph $\Gamma'$ describing $M$ can be achieved from a 2-regular admissible semi-free multigraph followed by the 4 operations in Theorem \ref{t11}. Since $\textrm{Todd}(M)=1$ by Lemma \ref{todd}, Theorem \ref{t21} implies that $\Gamma'$ has exactly 1 vertex of index 0. Since the 4 operations in Theorem \ref{t11} do not change the number of vertices (fixed points) of index 0, to achieve $\Gamma'$ from an admissible semi-free multigraph by the 4 operations in Theorem \ref{t11}, we must begin with a single 2-regular admissible semi-free multigraph $\Gamma$ with $T(\Gamma)=1$ (see Figure \ref{fig6}), that is, $\Gamma$ has precisely one vertex of index 0. \end{proof}

Moreover, the converse of Theorem \ref{t12} holds.

\begin{theo} \label{t83}
Let $\Gamma$ be a multigraph obtained from a 2-regular admissible semi-free multigraph $\Gamma'$ with $T(\Gamma')=1$ (Figure \ref{fig6}) followed by a combination of the 4 operations in Theorem \ref{t11}. Then there exists a 4-dimensional compact connected K\"ahler manifold equipped with a $\mathbb{T}^2$-action such that $\Gamma$ describes $M$ for the action of the first $S^1$-factor. 
\end{theo}

\begin{proof}
By Lemma \ref{extendcombine}, $\Gamma$ extends to $\mathbb{Z}^2$. Since the 4 operations in Theorem \ref{t11} do not change the number of vertices of index 0, the number $T(\Gamma)$ of vertices of index 0 is 1. Then this theorem follows from Proposition \ref{p37}.
\end{proof}

Let the circle act on a 4-dimensional compact connected manifold, complex or symplectic, preserving the complex or symplectic structure. Assume that the fixed point set is non-empty and discrete. Carrell, Howard, and Kosniowski for a complex manifold \cite{CHK} and Karshon \cite{Ka} for a symplectic manifold proved that such a manifold can be obtained from $\mathbb{CP}^2$ or a Hirzebruch surface followed by blow ups, where any blow up only occurs at a fixed point of index 1. Let $p$ be a fixed point of index 1. Let $\{-a,b\}$ be the weights at $p$ for some positive integers $a$ and $b$. Blowing up $p$ equivariantly (either in complex or symplectic sense), we get two fixed points $p'$ and $p''$ that have weights $\{-a,a+b\}$ and $\{-a-b,b\}$ respectively, instead of $p$. For a multigraph describing a manifold, blowing up at a fixed point of index 1 corresponds to Operation 1 of Theorem \ref{t11}; see Figure \ref{fig2}. To compare our minimal model (Figure \ref{fig6}) and their minimal models (Figures \ref{fig7-1}, \ref{fig7-2}, \ref{fig7-3}), we first describe circle actions on $\mathbb{CP}^2$ and Hirzebruch surfaces. 

\begin{exa} Let $S^1$ act on $\mathbb{CP}^2$ by
\begin{center}
$g \cdot [z_0:z_1:z_2]=[z_0:g^a z_1:g^{a+b} z_2]$
\end{center}
for all $g \in S^1 \subset \mathbb{C}$, for some positive integers $a$ and $b$ that are relatively prime. The action has three fixed points $p_1=[1:0:0]$, $p_2=[0:1:0]$, and $p_3=[0:0:1]$, and the weights at the fixed points are $\{a+b,a\}$, $\{-a,b\}$, and $\{-b,-a-b\}$, respectively. Figure \ref{fig7-1} is the multigraph describing $\mathbb{CP}^2$. \end{exa}

\begin{exa} \label{e210}
Let $n$ be an integer. A Hirzebruch surface is defined as 
\begin{center}
$\{([z_0:z_1:z_2],[w_1:w_2])\in\mathbb{CP}^2\times\mathbb{CP}^1|z_1 w_2^n=z_2 w_1^n\}$. 
\end{center}
Let $c$ and $d$ be positive integers such that $c-nd \neq 0$. Let $S^1$ act on the Hirzebruch surface by
\begin{center}
$g \cdot ([z_0:z_1:z_2],[w_1:w_2]) = ([g^c z_0:z_1:g^{nd} z_2],[w_1:g^d w_2])$
\end{center}
for all $g \in S^1 \subset \mathbb{C}$.
The action has 4 fixed points, $p_1=([1:0:0],[1:0]), p_2=([1:0:0],[0:1]), p_3=([0:1:0],[1:0])$, and $p_4=([0:0:1],[0:1])$, and the weights at the fixed points are $\{-c,d\}$, $\{nd-c,-d\}$, $\{c,d\}$, and $\{c-nd,-d\}$, respectively. If $c-nd$ is positive, Figure \ref{fig7-2} with $e=c-nd$ describes the manifold. Suppose that $c-nd$ is negative. In this case, if $d=1$, Figure \ref{fig7-2} with $e=nd-c$ describes the manifold, and if $d>1$, Figure \ref{fig7-3} with $e=nd-c$ describes the manifold. If $d=1$, Figure \ref{fig7-3} also describes the manifold, but we choose Figure \ref{fig7-2} to describe the manifold so that the multigraph can be achieved from our minimal model. This results from the fact that any multigraph that we associate in this paper satisfies the conditions in Lemma \ref{l26}; since $d=1$ is the smallest positive weight, for any edge $e$ with label 1, a multigraph we associate satisfies $n_{t(e)}=n_{i(e)}+1$ where $n_v$ is the index of a vertex $v$; see Definition \ref{d11}. Note that when Figure \ref{fig7-2} describes the manifold, $c \equiv e \mod d$ and when Figure \ref{fig7-3} describes the manifold, $c \equiv -e \mod d$. \end{exa}

\begin{figure}
\centering
\begin{subfigure}[b][7cm][s]{.25\textwidth}
\centering
\vfill
\begin{tikzpicture}[state/.style ={circle, draw}]
\node[state] (a) {$p_1$};
\node[state] (b) [above right=of a] {$p_2$};
\node[state] (c) [above left=of b] {$p_3$};
\path (a) [->] edge node[right] {$a$} (b);
\path (b) [->] edge node[right] {$b$} (c);
\path (a) [->]edge node [right] {$a+b$} (c);
\end{tikzpicture}
\vfill
\caption{$\mathbb{CP}^2$}\label{fig7-1}
\vspace{\baselineskip}
\end{subfigure}\qquad
\begin{subfigure}[b][7cm][s]{.25\textwidth}
\centering
\vfill
\begin{tikzpicture}[state/.style ={circle, draw}]
\node[state] (A) {$p_1$};
\node[state] (B) [above left=of A] {$p_2$};
\node[state] (C) [above right=of A] {$p_3$};
\node[state] (D) [above right=of B] {$p_4$};
\path (A) [->] edge node[right] {$d$} (B);
\path (A) [->] edge node [right] {$c$} (C);
\path (B) [->] edge node [right] {$e$} (D);
\path (C) [->] edge node [right] {$d$} (D);
\end{tikzpicture}
\vfill
\caption{$c-nd>0$ or $d=1$}\label{fig7-2}
\vspace{\baselineskip}
\end{subfigure}\qquad
\begin{subfigure}[b][7cm][s]{.25\textwidth}
\centering
\vfill
\begin{tikzpicture}[state/.style ={circle, draw}]
\node[state] (A) {$p_1$};
\node[state] (B) [above left=of A] {$p_3$};
\node[state] (C) [above =of B] {$p_4$};
\node[state] (D) [above right=of C] {$p_2$};
\path (A) [->] edge node[right] {$c$} (B);
\path (A) [->] edge node [right] {$d$} (D);
\path (B) [->] edge node [right] {$d$} (C);
\path (C) [->] edge node [left] {$e$} (D);
\end{tikzpicture}
\vfill
\caption{$c-nd<0$ and $d>1$}\label{fig7-3}
\end{subfigure}
\caption{Multigraphs for $\mathbb{CP}^2$ and Hirzebruch surface}\label{fig7}
\end{figure}
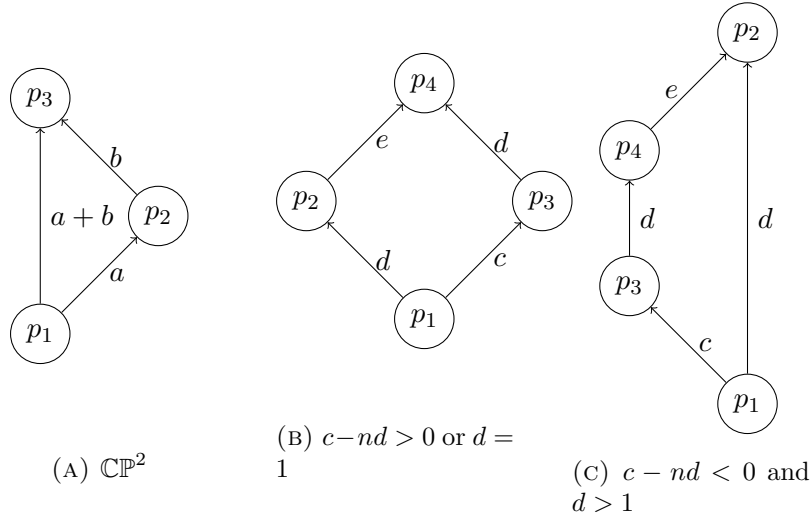

On the other hand, as in Theorem \ref{t11}, our minimal model is described by a semi-free multigraph. Therefore, we shall discuss how to go from our minimal multigraph Figure \ref{fig6} to their minimal multigraphs Figure \ref{fig7-1}, Figure \ref{fig7-2}, and Figure \ref{fig7-3} with $d>1$, by our operations in Theorem \ref{t11}.

For a multigraph Figure \ref{fig7-1} for $\mathbb{CP}^2$, perform the reverse of Operation 4 in Theorem \ref{t11}. Then the resulting multigraph is Figure \ref{fig7-2} if we set $a=c$, $b=d$, and $e=c$. This means that if we can go from our minimal multigraph Figure \ref{fig6} to Figure \ref{fig7-2} with $a=c$, $b=d$, and $e=c$ by our operations, then we can perform Operation 4 in Theorem \ref{t11} to Figure \ref{fig7-2} with $a=c$, $b=d$, and $e=c$ to reach Figure \ref{fig7-1}. Therefore, for a multigraph for $\mathbb{CP}^2$, our task is to show how to go from Figure \ref{fig6} to Figure \ref{fig7-2}.

From now on, we will show how to go from their minimal multigraphs to our minimal multigraph by performing reverses of our operations. Then one can chase the reversed operations and go from our minimal one to their minimal ones by performing our operations. Therefore, our task is to show how to go from Figure \ref{fig7-2} and Figure \ref{fig7-3} with $d>1$ to Figure \ref{fig6}.

We show how to go from Figure \ref{fig7-3} with $d>1$ to Figure \ref{fig6}. In Figure \ref{fig7-3}, by performing the reverse of Operation 2 $j_1$-times, replace the label $c$ of $(p_1,p_3)$-edge by $c-j_1d$ where $j_1$ is the biggest non-negative integer such that $c-j_1d$ is positive and call it $c_1$. Then $c_1<d$. Similarly, by performing the reverse of Operation 3 $j_2$-times, replace the label $e=nd-c$ of $(p_4,p_2)$-edge by $nd-c-j_2d$ where $j_2$ is the biggest non-negative integer such that $nd-c-j_2d$ is positive and call it $c_2$. Then $c_2<d$. Since $c_1=c-j_1d$, $c_2=nd-c-j_2d$, $0<c_1<d$, $0<c_2<d$, and $c_1+c_2$ is a multiple of $d$, we have $c_1+c_2=d$. The resulting multigraph is Figure \ref{fig8-1}. Next, do the reverse of Operation 4 for $(p_1,p_2)$-edge. The resulting multigraph is Figure \ref{fig8-2}. Next, do the reverse of Operation 1 for $(p_3,p_4)$-edge. Then the resulting multigraph is Figure \ref{fig8-3}.

Both of Figure \ref{fig7-2} and Figure \ref{fig8-3} with $d>1$ are now included in Figure \ref{fig9} with $f \equiv h \mod g$ if we relabel the labels suitably, where $f$, $g$, and $h$ are pairwise prime. Therefore, it remains to show how to go from Figure \ref{fig9} to Figure \ref{fig6}.

Now, by performing the reverse of Operation 2 for $(p_1,p_2)$-edge $k_1$-times, replace the label $f$ of $(p_1,p_2)$-edge by $f-k_1g$ where $k_1$ is the biggest non-negative integer such that $f-k_1g$ is positive and call it $f_1$. Similarly, by performing the reverse of Operation 3 for $(p_3,p_4)$-edge $k_2$ times, replace the label $h$ of $(p_3,p_4)$-edge by $h-k_2g$ where $k_2$ is the biggest non-negative integer such that $h-k_2g$ is positive and call it $h_1$. Since $f \equiv h \mod g$, this means that $f_1=h_1$ and $f_1<g$. Next, do the reverse of Operation 3 for $(p_2,p_4)$-edge $k_3$ times to replace the label $g$ of the $(p_2,p_4)$-edge by $g-k_3f_1$ where $k_3$ is the biggest non-negative integer such that $g-k_3f_1$ is positive and call it $g_1$. Do the reverse of Operation 2 for $(p_1,p_3)$-edge $k_4$ times to replace the label $g$ of the $(p_1,p_3)$-edge by $g-k_4f_1$ where $k_4$ is the biggest non-negative integer such that $g-k_4f_1$ is positive and call it $g_2$. Then we have that $g_1=g_2$ and $g_1<f_1$. Then we are still in the case of Figure \ref{fig9}, but now the labels of all edges are reduced. Therefore, if we repeat this process, we can reduce the label of every edge to 1. Hence we can reach our multigraph Figure \ref{fig6}. This means that from Figure \ref{fig6} we can reach Figure \ref{fig7-2} and Figure \ref{fig8-3} (with $d>1$) and hence Figure \ref{fig7-1} by performing our Operations 1, 2, 3 and 4 in Theorem \ref{t11}.

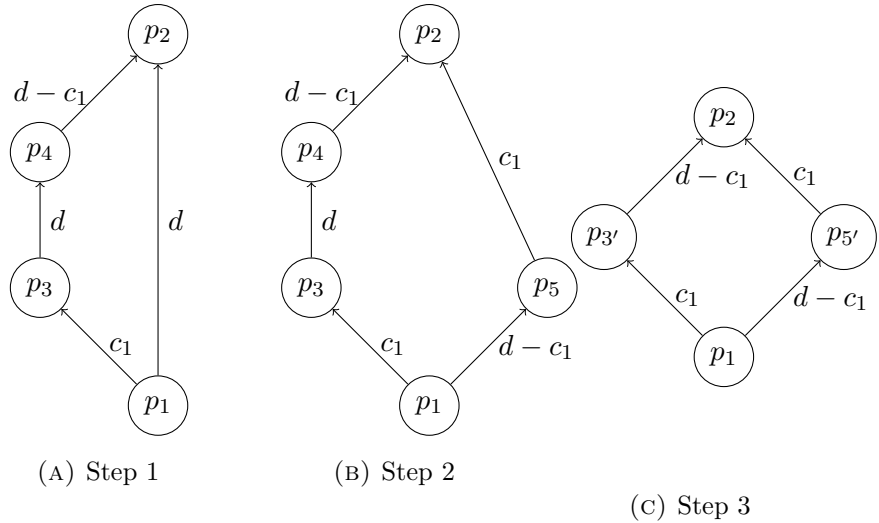
\begin{figure}
\centering
\begin{subfigure}[b][7cm][s]{.25\textwidth}
\centering
\vfill
\begin{tikzpicture}[state/.style ={circle, draw}]
\node[state] (A) {$p_1$};
\node[state] (B) [above left=of A] {$p_3$};
\node[state] (C) [above =of B] {$p_4$};
\node[state] (D) [above right=of C] {$p_2$};
\path (A) [->] edge node[right] {$c_1$} (B);
\path (A) [->] edge node [right] {$d$} (D);
\path (B) [->] edge node [right] {$d$} (C);
\path (C) [->] edge node [left] {$d-c_1$} (D);
\end{tikzpicture}
\vfill
\caption{Step 1}\label{fig8-1} 
\vspace{\baselineskip}
\end{subfigure}\qquad
\begin{subfigure}[b][7cm][s]{.25\textwidth}
\centering
\vfill
\begin{tikzpicture}[state/.style ={circle, draw}]
\node[state] (A) {$p_1$};
\node[state] (B) [above left=of A] {$p_3$};
\node[state] (C) [above =of B] {$p_4$};
\node[state] (D) [above right=of C] {$p_2$};
\node[state] (E) [above right=of A] {$p_5$};
\path (A) [->] edge node[right] {$c_1$} (B);
\path (A) [->] edge node [right] {$d-c_1$} (E);
\path (B) [->] edge node [right] {$d$} (C);
\path (C) [->] edge node [left] {$d-c_1$} (D);
\path (E) [->] edge node [right] {$c_1$} (D);
\end{tikzpicture}
\vfill
\caption{Step 2}\label{fig8-2} 
\vspace{\baselineskip}
\end{subfigure}\qquad
\begin{subfigure}[b][7cm][s]{.25\textwidth}
\centering
\vfill
\begin{tikzpicture}[state/.style ={circle, draw}]
\node[state] (A) {$p_1$};
\node[state] (B) [above left=of A] {$p_{3'}$};
\node[state] (C) [above right=of A] {$p_{5'}$};
\node[state] (D) [above right=of B] {$p_2$};
\path (A) [->] edge node[right] {$c_1$} (B);
\path (A) [->] edge node [right] {$d-c_1$} (C);
\path (B) [->] edge node [right] {$d-c_1$} (D);
\path (C) [->] edge node [right] {$c_1$} (D);
\end{tikzpicture}
\vfill
\caption{Step 3}\label{fig8-3} 
\end{subfigure}
\caption{Going from Figure \ref{fig7-3} to Figure \ref{fig9}}\label{fig8}
\end{figure}

\begin{figure}
\centering
\begin{subfigure}[b][4cm][s]{.25\textwidth}
\centering
\vfill
\begin{tikzpicture}[state/.style ={circle, draw}]
\node[state] (A) {$p_1$};
\node[state] (B) [above left=of A] {$p_2$};
\node[state] (C) [above right=of A] {$p_3$};
\node[state] (D) [above right=of B] {$p_4$};
\path (A) [->] edge node[right] {$f$} (B);
\path (A) [->] edge node [right] {$g$} (C);
\path (B) [->] edge node [right] {$g$} (D);
\path (C) [->] edge node [right] {$h$} (D);
\end{tikzpicture}
\vfill
\end{subfigure}
\caption{Can go from Figure \ref{fig6} to this.}\label{fig9}
\end{figure}
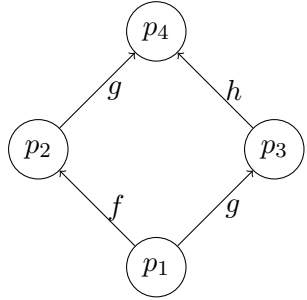

\section{Semi-free actions with discrete fixed point sets} \label{s6}

In this section, we discuss semi-free circle actions with discrete fixed point sets on symplectic manifolds and almost complex manifolds, and prove Theorem \ref{t13}. Recall that Theorem \ref{t13} provides a necessary and sufficient condition for the numbers $N_i$ of fixed points of index $i$ for such an almost complex manifold. As mentioned in the introduction, Tolman and Weitsman \cite{TW} proved that a semi-free symplectic circle action with a discrete fixed point set is Hamiltonian if and only if the fixed point set is non-empty. Feldman \cite{Fe} and Li \cite{L} reproved this.

\begin{theo} \label{t51} \cite{Fe}, \cite{L}, \cite{TW} Let the circle act symplectically and semi-freely on a compact connected symplectic manifold $M$ with a discrete fixed set. If the fixed point set is non-empty, the action is Hamiltonian. Moreover, there are precisely $2^n$ fixed points, where $2n=\dim M$. \end{theo}

To prove Theorem \ref{t51}, Tolman and Weitsman proved the following lemma by using equivariant cohomology and the Atiyah-Bott-Berline-Vergne localization formula. This is also proved in \cite{J3} by using Theorem \ref{t21}.

\begin{lem} \label{l23} \cite{J3}, \cite{TW} Let the circle act semi-freely on a $2n$-dimensional compact almost complex manifold with a discrete fixed point set. Then for $0 \leq i \leq n$, $N_i=N_0 \cdot {n \choose i}$, where $N_i$ is the number of fixed points of index $i$. \end{lem}

For a symplectic action there are exactly $2^n$ fixed points, since the fixed point set is discrete and there is a unique fixed point of index 0, that is, $N_0=1$.

For any possible fixed point data for a semi-free $S^1$-action on a compact almost complex manifold with a discrete fixed point set, we construct an almost complex $S^1$-manifold that has the given fixed point data. 

\begin{theo} \label{t52} Let $n \geq 2$. Given any positive integer $k$, there exists a $2n$-dimensional compact connected almost complex manifold $M$ equipped with a semi-free circle action having a discrete fixed point set, whose Todd genus is $k$. Moreover, for $0 \leq i \leq n$, $N_i=k \cdot {n \choose i}$, where $N_i$ is the number of fixed points of index $i$. In particular, the total number of fixed points is equal to $k \cdot 2^n$. \end{theo}

\begin{proof} Let $\Gamma$ be a connected 2-regular admissible semi-free multigraph with $T(\Gamma)=k$. By Lemma \ref{l32}, $\Gamma$ extends to $\mathbb{Z}^2$. Hence, by Proposition \ref{p37}, there exists a 4-dimensional compact connected almost complex manifold $M'$ equipped with a semi-free circle action having a discrete fixed point set that is described by $\Gamma$. The manifold $M'$ has $4k$ fixed points, $k$ fixed points of index 0, $2k$ fixed points of index 1, and $k$ fixed points of index 2. Let $M''$ be the product $S^2 \times \cdots \times S^2$ of $(n-2)$ copies of the 2-spheres, where on each $S^2$ the circle acts by rotating once. The rotation on each 2-sphere has 2 fixed points, the north pole and the south pole, that have weights $-1$ and $+1$, respectively. Thus, $M''$ has $2^{n-2}$ fixed points, and the number of fixed points with index $i$ is equal to ${n-2 \choose i}$ for each $0 \leq i \leq n-2$. Taking the product of $M'$ and $M''$ with a diagonal action, the desired manifold $M$ is constructed. Since $M'$ has $k$ fixed points of index 0 and $M''$ has 1 fixed point of index 0, it follows that $M$ has $k$ fixed points of index 0. Since the number of fixed points of index 0 is equal to the Todd genus of $M$ by Theorem \ref{t21}, this theorem follows. \end{proof}

Theorem \ref{t13} follows from Lemma \ref{l23} and Theorem \ref{t52}.

\section*{Funding}

This work was supported by the National Research Foundation of Korea(NRF) grant funded by the Korea government(MSIT) [2021R1C1C1004158 to D.J.].

\section*{Acknowledgments}

The author would like to thank Mikiya Masuda and Susan Tolman for fruitful conversations. The author would also like to thank the anonymous referee for many valuable comments and suggestions which helped improve the exposition and quality of this paper.


\begin{thebibliography}{1}

\bibitem[AH]{AhHa}
K. Ahara and A. Hattori: \emph{4 dimensional symplectic $S^1$-manifolds admitting moment map.} J. Fac. Sci. Univ. Tokyo Sect. IA, Math. \textbf{38} (1991), 251-298.

\bibitem[A]{Au}
M. Audin: \emph{Hamiltoniens p\'eriodiques sur les vari\'et\'es symplectiques compactes
de dimension 4}. G\'eom\'etrie symplectique et m\'ecanique, Proceedings 1988, C.
Albert ed., Springer Lecture Notes in Math. \text{1416} (1990).

\bibitem[CHK]{CHK}
J. Carrell, A. Howard, and C. Kosniowski: \emph{Holomorphic vector fields on complex surfaces.} Math. Ann. \textbf{204} (1973) 73-81.

\bibitem[Fe]{Fe}
K. E. Feldman: \emph{Hirzebruch genus of a manifold supporting a Hamiltonian circle action}. Russian Math. Surveys \textbf{56} (2001) 978-979.

\bibitem[Fi]{F}
R. Fintushel: \emph{Classification of circle actions on 4-manifolds}. Trans. Amer. Math. Soc. \textbf{232} (1978) 377-390.

\bibitem[GS]{GS}
L. Godinho and S. Sabatini: \emph{New tools for classifying Hamiltonian circle actions with isolated fixed points}, Found. Comput. Math. \textbf{14} (2014), 791-860.

\bibitem[HBJ]{HBJ}
F. Hirzebruch, T. Berger, and R. Jung: \emph{Manifolds and modular forms}. Aspects of Mathematics, vol. E20, Vieweg, Braunschweig, (1992).

\bibitem[H]{H}
A. Hattori: \emph{$S^1$-actions on unitary manifolds and quasi-ample line bundles}, J. Fac. Sci. Univ. Tokyo Sect. IA Math. \textbf{31} (1985), 433-486.

\bibitem[JT]{JT}
D. Jang and S. Tolman: \emph{Hamiltonian circle actions on eight dimensional manifolds with minimal fixed sets.} Transform. Groups., \textbf{22} (2017), 353-359.

\bibitem[J1]{J1}
D. Jang: \emph{Symplectic periodic flows with exactly three equilibrium points}. Ergodic Theory Dynam. Systems \textbf{34} (2014), 1930-1963.

\bibitem[J2]{J2}
D. Jang: \emph{Circle actions on almost complex manifolds with isolated fixed points.} J. Geom. Phys., \textbf{119} (2017), 187-192.

\bibitem[J3]{J5}
D. Jang: \emph{Symplectic circle actions with isolated fixed points.} J. Symplectic Geom. \textbf{15} (2017), 1071-1087.

\bibitem[J4]{J4}
D. Jang: \emph{Circle actions on oriented manifolds with discrete fixed point sets and classification in dimension 4}. J. Geom. Phys., \textbf{133} (2018), 181-194.

\bibitem[J5]{J3}
D. Jang: \emph{Circle actions on almost complex manifolds with 4 fixed points}. Math. Z. \textbf{294} (2020), 287-319.

\bibitem[K]{Ka}
Y. Karshon: \emph{Periodic Hamiltonian flows on four dimensional manifolds.} Mem. Amer.
Math. Soc. \textbf{672} (1999).

\bibitem[L]{L}
P. Li: \emph{The rigidity of Dolbeault-type operators and symplectic circle actions.} Proc. Amer. Math. Soc. \textbf{140} (2012) 1987-1995.

\bibitem[Ma]{M}
M. Masuda: \emph{Unitary toric manifolds, multi-fans and equivariant index}. Tohoku. Math. J. \textbf{51} (1999) 237-265.

\bibitem[Mc]{MD}
D. McDuff: \emph{The moment map for circle actions on symplectic manifolds.} J. Geom. Phys. \textbf{5} (1988), no. 2, 149-160.

\bibitem[S]{S}
S. Sabatini: \emph{On the Chern numbers and the Hilbert polynomial of an almost complex manifold with a circle action.} Commun. Contemp. Math., \textbf{19}, 1750043 (2017).

\bibitem[TW]{TW}
S. Tolman and J. Weitsman: \emph{On semifree symplectic circle actions with isolated fixed points.} Topology \textbf{39} (2000), no. 2, 299-309.

\bibitem[T]{T}
S. Tolman: \emph{On a symplectic generalization of Petrie's conjecture.} Trans. Amer. Math. Soc. \textbf{362} (2010), no. 8, 3963-3996.
\end{thebibliography}
\end{document}